\documentclass{article}
\usepackage[spanish,english,activeacute]{babel}
\usepackage[english,activeacute]{babel}
\usepackage[latin1]{inputenc}
\usepackage[T1]{fontenc}
\usepackage{ae,aecompl}
\usepackage{amsmath,amssymb,amsfonts,latexsym}
\usepackage{amsthm}
\usepackage{bm}
\usepackage{txfonts}

\usepackage{enumerate}
\usepackage[all]{xy}
\usepackage{graphicx}
\usepackage{color}
\usepackage{colortbl}
\usepackage{epstopdf,epsfig,epsf}
\usepackage[section]{placeins}
\usepackage{subfig}
\usepackage{flafter}
\usepackage[bookmarks=true,colorlinks=true,linkcolor=blue,citecolor=green,urlcolor=red,plainpages=false,pdftex]{hyperref}
\usepackage{layout} 
\theoremstyle{plain} 
 \newtheorem{theorem}{Theorem}[section]   
 \newtheorem{proposition}[theorem]{Proposition}

\theoremstyle{definition} 

 \newtheorem{example}{Example}[section]

\theoremstyle{remark} 

\numberwithin{equation}{section}  
\numberwithin{figure}{section}
\numberwithin{table}{section}
\allowdisplaybreaks 

 \title{\textbf{Polar differentiation matrices for the Laplace equation  in the disk subjected to nonhomogeneous Dirichlet, Neumann and Robin boundary conditions and the biharmonic equation subjected to nonhomogeneous Dirichlet conditions }}
 
 \author{Marcela Molina-Meyer$^{1}$, Frank Richard Prieto Medina\\
Universidad Carlos III de Madrid, Departamento de Matem\'aticas\\
Avenida de la Universidad 30, 28911, Legan\'es, Spain}

\date{\today}

\begin{document}



 \maketitle
 \begin{abstract}
 
In this paper we present a pseudospectral method in the disk.  Unlike the methods known until now, the disk  is not duplicated. Moreover, we solve  the Laplace equation subjected to nonhomogeneous Dirichlet, Neumann and Robin boundary conditions and the biharmonic equation subjected to nonhomogeneous Dirichlet conditions by only using the elements of the corresponding  differentiation matrices. It is worth noting that we don not  use any quadrature, do not need to solve any decoupled system of ordinary differential equations, do not use any pole condition and do not require any lifting. We  solve several numerical examples showing that the spectral convergence is being met. The pseudospectral method developed in this paper can be applied to estimate Sherwood numbers integrating the mass flux to the disk   and it can be easily implemented to solve Lotka-Volterra systems and  nonlinear problems involving chemical reactions.

\end{abstract}



\footnotetext[1]{Corresponding author. 
E-mail address: mmolinam@math.uc3m.es (M. Molina-Meyer).}


\newpage

\section{Introduction}           
The Laplace operator is widely used in mathematical models of macroscopic chemotaxis, hydrodynamic,  semiconductors, mass transfer, growth of species or in the research and development of new acoustic and optical  instruments. See   \cite{markovich} and \cite{metas}. Concurrently, the biharmonic operator is present in mathematical models of  elasticity, such as the flexure of thin plates, or  in the dynamics of bio-fluids,  such as arterial blood flows. See \cite{HughesMarsden},  \cite{Sweers}  and \cite{Selvaduarai}.  In electrochemical experiments, the diffusion coefficient is determine using a rotating disk electrode tecnique by measuring Sherwood numbers,  \cite{Igor}.  The pseudospectral method developed in this paper can be applied to estimate Sherwood numbers integrating the mass flux to the disk, \cite{rotatingdisk} and \cite{Igor}. Moreover, it can be implemented to solve Lotka Volterra systems, \cite{JLGMMM},  and  nonlinear problems  involving chemical reactions, \cite{finlayson}.

 Sometimes, as in \cite{JLGEDM}, \cite{Rosa} , \cite{JLGMMM}, \cite{JLGMMMiso},  \cite{Hsu}, \cite{JLGMMMTELL} and \cite{LGMMR} the simulations of solutions of some non-linear equations and non-linear systems allow us to conjecture open problems.  In fact  more realistic mathematical models, of engineering problems, ecological and biological phenomena, can be derived by using variable coefficients, nonlinear terms and non-homogeneous boundary conditions. 
In \cite{Rosa}, \cite{JLGMMM}, \cite{JLGEDM}, \cite{JLGMMMiso}  and \cite{JLGMMMTELL}, Fourier pseudospectral methods are  used, and in \cite{Hsu} and \cite{LGMMR}, Chebyshev pseudospectral methods are applied. Very recently, \cite{young}  reviewed  the treatment of boundary conditions involving fluxes in orthogonal collocation methods. Although one dimensional domains are considered in all these papers. 

Nowadays it is necessary to develop efficient and accurate numerical methods to finely analyze the behavior of some non-radially symmetric solutions of two dimensional linear and non-linear equations involving the Laplace and the biharmonic  operators.   In this respect, the differentiation matrices obtained in this paper allow to calculate the numerical  solutions in the disk subject to all types  of non-homogeneous boundary conditions, whether they are Dirichlet, Neumann or Robin. Moreover, this paper offers all the calculations needed to solve the Laplace and the biharmonic nonhomogeneous  equations by only using the elements of the differentiation matrices.  

Unfortunately, the methods used  in \cite{BernardidiscoNavier}, \cite{BernardidiscoLaplace}, \cite{shen1}, \cite{shen2} \cite{Treflibro}, \cite{Forpolar}, \cite{HWK} and  \cite{Towsendolver} can only be applied in case of homogeneous Dirichlet boundary conditions. All this papers propose to use a lifting in case of nonhomogeneous Dirichlet boundary conditions.  In fact, none of these references solve problems subject to  Neumann or Robin boundary conditions. Using a lifting has many disadvantages, it is necessary to calculate it beforehand because it is needed  to reformulate the original problem, it implies that certain conditions of smoothness on the boundary conditions must be assumed and in the case of having boundary conditions provided by a table, these data must first be interpolated. Hence, using a lifting significantly increases the computational cost. However, the pseudo-spectral method presented in this paper require no lifting as we compute  the polar  differentiation matrices differentiating the interpolation  polynomial in the disk that satisfies the nonhomogeneous boundary conditions. In conclusion, our method is a direct method with lower computational cost.

Furthermore, collocation methods are well known because of their advantages:  they are direct and easy to implement and  in the case of Chebyshev Gauss Lobatto (CGL) collocation points, if the data are sufficiently smooth, the approximate solution has spectral accuracy. In \cite{R}, \cite{CMAT}, \cite{QV}, \cite{Shenlibro} and \cite{CY} the  convergence and stability  of the  collocation method are demonstrated in cases where the discrete bilinear form is exact and the collocation method matches a Galerkin method.

 In fact, in \cite{Towsendolver} and \cite{DriscollHale} to incorporate the boundary  conditions some rows of the matrix obtained by the Tau method are removed.  Unfortunately,  excluding  rows eliminates  some projections of the best approximation whose consequence could be a drastic  undesired change in the numerical solution.  In addition, in \cite{DriscollHale} the interpolation polynomial of CGL points (extrema of the first-kind Chebyshev polynomial) satisfies the boundary conditions, but the equation is asked to be satisfied in a Chebyshev Gauss grid (roots of the first-kind Chebyshev polynomial) of a lower order. Therefore, the resulting  differentiation matrices are  rectangular and do not correspond to any discrete Galerkin method. Moreover, a fictitious point outside the domain is also introduced in \cite{For} resulting in an unstable method according to \cite{DriscollHale}.

In this paper, we propose a method that does not require any pole condition. Unfotunately,  \cite{shen1}, \cite{shen2} \cite{HWK} and \cite{Towsendwright} apply a Fourier Galerkin method which results in a decoupled system of boundary value problems where pole conditions need to be imposed. In particular, \cite{HWK} uses a collocation method for each boundary value problem.

Nevertheless, in \cite{Towsendwright} the  Fourier Chebyshev spectral method is applied in a rectangular domain that corresponds to repeating the disk twice and    the solution should finally be restricted to the sector of the rectangle that corresponds to the positive radii. In \cite{BayonaFonberg} are considered fictitious points outside the disk, but the equation must be satisfied on the boundary what distorts the original problem. Many times the solution does not have the sufficient regularity on the boundary to be able to apply the operator of partial differential equations.  Moreover, in \cite{Trefftz} and \cite{CY} as a consequence of applying Gaussian quadrature, two separated sets of weights are required, one in the interior of the domain and one on the boundary. 

Even more, the importance of the polar differentiation matrices could be inferred from the commentary ''{\it One needs a Fourier Galerkin-Chebyshev collocation method} '' in  Section 3.9 of \cite{CMAT}. To deduce them, we first derived the  trigonometric polynomial corresponding to each concentric circle of radius equal the CGL positive points.  Then, using Corollary 1.47 and Theorem 1.4.2 in \cite {Stenger}, due to the smoothness of the solution and the properties of Dirichlet kernel, we proved that the above interpolation polynomial coincides with the approximate solution proposed in \cite{WD}. Thereof, following the former results of polar sampling in \cite{Stark} and in \cite{libro1},  we obtained the positive CGL points in the radial coordinate. At this time, it should be noted that considering only  positive radii is not an original idea of  \cite{Forpolar}, but to \cite{Marvasti}.

Now, to start with the  Laplace and biharmonic polar differentiation matrices we  introduce  the collocation points in the disk
 
 \begin{equation} \label{puntoscolocacion}
(r_k,\theta_l)= \left(R \,cos ( \dfrac{(k-1) \> \pi}{N_r}), \dfrac{2\pi \> l }{N_{\theta}} \right), \quad 1 \leq k \leq \frac{N_r+1}{2} , \quad 1 \leq l \leq {N_{\theta}}.  
\end{equation}

\noindent Thereupon, from the symmetry property  
\begin{equation} \label{simetria}
u (r_{N_r+2-i},\theta_j)= u(r_i,\theta_{j+\frac{N_\theta}{2}}), \quad 1 \leq j \leq \frac {N_{\theta}}{2}, \quad 1 \leq i \leq \frac{N_r+1}{2},   
\end{equation} 
defined in  \cite{Marvasti}, we obtain the  interpolation polynomial in the disk
\begin{eqnarray} \label{numeincogintro}
u_{\frac{N_r+1}{2},N_{\theta}}(r,\theta) & = & \sum_{k=1}^\frac{N_r+1}{2} \sum_{l=1}^{N_{\theta}} a_{k,l} \left[ S_{N_{\theta}}(\theta-\theta_l) L_k(r) + S_{N_{\theta}}(\theta-\theta_{l+\frac{N_\theta}{2}}) L_{N_r+2-k}(r) \right],
\end{eqnarray}
where  $a_{k,l}=u(r_k,\theta_l)$,   $
S_{N_{\theta}}(\theta) = \dfrac{sin\left( \dfrac{N_{\theta} \theta}{2}\right)}{N_{\theta}\, tan\left( \dfrac{\theta}{2}\right)}$ 
and $L_k$'s are the corresponding Lagrange polynomials.
In particular, $N_{r}$ must be an odd number to avoid the origin being a collocation point and  $N_{\theta}$ must be an even number to be able to apply the properties of the Dirichlet kernel. Specifically, the existence and uniqueness of Fourier Chebyshev interpolation polynomials in the disk were first proved in \cite{Blending} and \cite {shen1}. According to the information at our disposal, the expression  (\ref{numeincogintro}) of the interpolation polynomial in the disk has been obtained for the first time in this paper. Concretely, we obtained the interpolation polynomial  \eqref{numeincogintro} in the disk with a total of  ${\frac {N_r+1} 2} \times N_\theta$  unknown coefficients, corresponding to the values of the numerical solution in the collocation points defined in \eqref{puntoscolocacion}.  Unlike the methods known so far, the disk is not duplicated. We should note here that the first ideas on polar differentiation matrices were developed in \cite{TFMFrank}. 

Thereupon to obtain  polar differentiation matrices we proceeded  in five steps.  First, we imposed that $u_{\frac{N_r+1}{2},N_{\theta}}(r,\theta)$ satisfies the boundary conditions.   Second, we cleared from the equations obtained above, in the case of the Laplace equation, all the values of $u_{\frac{N_r+1}{2},N_{\theta}}(r,\theta)$ on the boundary  and, in the case of the biharmonic equation,  all the values of the two outer circles.  Third, we substituted  all these boundary values in  \eqref{numeincogintro}. Fourth, we applied the Laplace and biharmonic operators in the remaining interior collocation points, respectively. And finally, we developed both operations on block matrices  and Kronecker products obtaining a smaller and less ill conditioned system. 

Moreover, the deduced linear systems have  smaller effective condition numbers, see \cite {Trefftz}. In particular, a finite difference preconditioner for a Fourier-Chebyshev collocation method was  developed  in \cite{WD}, even though in our case it is not indispensable to use a preconditioner as the numerical solutions achieve rapid or spectral convergence. Note that there is no preconditioner used in  \cite{Weideman}, \cite{JLGEDM}, \cite{Treflibro}, \cite{For} or  \cite{Hsu}.

Remarkably, even though there exists no explicit solution for the cases of  piece-wise constant  boundary conditions   of the Laplace equation in the disk, we can accurately calculate the numerical  solution and its convergence can be checked using Poisson's formula. Moreover, despite the fact that there is also no  explicit solution of the biharmonic equation in the disk for piecewise constant boundary conditions, to use the Green Function in \cite{Sweers} could provide an interesting test to verify  the convergence of the numerical solution.

It is noteworthy, that this paper provides a finite rank approximation of the   resolvent operator associated with each boundary value problem whenever the collocation method matches with a Galerkin method, see \cite{Ahues}.  Notwithstanding that, this paper does not use any quadrature, it does not need to differentiate between weights on the boundary and the interior of the domain, it does not need to solve any  uncoupled system of ordinary differential equations and it does not require any  lifting.

So far, no explicit formulas of differentiation matrices associated with one dimensional boundary value problems subjected to nonhomogeneous Neumann or Robin boundary  conditions have been published in the literature, \cite{CY}, \cite{Treflibro},  \cite{For}, \cite{DriscollHale} and \cite{Trefblock} . In this paper, based on the ideas in  \cite{R}, we obtain explicit formulas for these cases. Moreover, through a new approach in which  $N+1$ CGL collocation points are used, we solved the biharmonic equation directly, both in an interval and in the disk. In the case of one-dimensional fourth order equations, as there are two conditions at each end point of the interval, we cleared the values of the interpolation polynomial in the points $x_1,x_2,x_N$ and $x_{N+1}$ in terms of the values of the approximate solution at the remaining inner points, obtaining a system of N-3 equations for the N-3 unknowns. Unfortunately, the  idea of \cite{FunaroHein} for homogeneous boundary conditions, that has been widely used in the literature to solve fourth order equations, see  \cite{Muiteref} and \cite{Treflibro}, can not be applied in the case of nonhomogeneous boundary conditions. Note that liftings are used in \cite{Weideman}. Moreover, for the biharmonic equation if CGL collocation points are considered, the continuous bilinear form is not equal anymore to the discrete bilinear form, which forces in \cite{CY} and \cite{FunaroHein} to choose as collocation points the zeros of the second derivative of the Chebyshev polynomial of order $N$. 

Finally to show how to use differentiation matrices in different types of problems, linear and non-linear, of  second or fourth order, in an interval or in a disk, in each section we have included illustrative numerical examples of each case, all of them showing rapid or exponential convergence.

This paper is organized as follows: Section 2 is concerned with second order one dimensional equations subjected to Dirichlet, Neumann and Robin nonhomogeneous boundary conditions and fourth order one dimensional equations subjected to Dirichlet nonhomogeneous boundary conditions. In Section 3, a detailed deduction of the interpolation polynomial in the disk is given, the Laplace differentiation matrices in polar coordinates are deduced and calculated, using Kronecker products and operations by blocks, for each nonhomogeneous Laplace equation, subjected to Dirichlet, Neumann and Robin nonhomogeneous conditions on the boundary. Lastly, the differentiation matrix for the nonhomogeneous biharmonic equation in the disk is thoroughly deduced and calculated.


\section{Differentiation matrices in one dimension}
\label{1dimension}

To describe our further results, we require some preliminaries about differentiation matrices. First, we consider the CGL nodes \begin{eqnarray}\label{CGL}
y_i  =  cos \left( \dfrac{(i-1) \> \pi}{N}\right), & & i=1,...,N+1.
\end{eqnarray}
as well as its related Lagrange polynomials  
\begin{eqnarray} \label{lag}
\widehat{L}_i(y)=\prod_{\substack{k=1 \\ k \neq i}}^{N+1} \left( \dfrac{y-y_k}{y_i-y_k}\right)
 & & i=1,...,N+1.
\end{eqnarray}
Similarly, we use (\ref{CGL}), we change the variables
\begin{eqnarray}\label{malladonuevo}
x  =  \dfrac{(b-a)\> y + b+a}{2},   
\end{eqnarray}
 and we define
\begin{eqnarray}\label{malladonuevo2}
x_i  =  \dfrac{(b-a)\> y_i + b+a}{2},  & & i=1,...,N+1,   
\end{eqnarray}
which yields to the following Lagrange polynomials 
\begin{eqnarray}
L_i(x)=L_i\left(\dfrac{(b-a)\> y + b+a}{2}\right)= \widehat{L}_i(y). &&
\end{eqnarray}

Thereupon, we consider the differential equation 
 \begin{eqnarray}\label{problemaunidimensional}
\dfrac{d^{\gamma}u}{dx^{\gamma}} + F(x,u) = 0, &  x\in(a,b),
\end{eqnarray} 


\noindent whose suitable regular solution  $u : [a,b] \rightarrow \mathbb{R}$ might satisfy either Dirichlet, Neumann or Robin homogeneous or nonhomogeneous boundary conditions. In particular,  $F(x,u)$ might be a linear or a non linear function and the index $\gamma$ might be either 2 or 4. It is the purpose of this article to approximate the solution of (\ref{problemaunidimensional}) by the interpolation polynomial $P_N(x)$ of $u(x)$ of degree $N$, satisfying $P_N(x_i)=u(x_i), \; i=1,...,N+1.$ Consequently, we define $u_N(x)$  as  
\begin{eqnarray}\label{PI}
u_N(x) = P_N(x) \coloneqq \sum_{i=1}^{N+1} a_i\> L_i(x).
\end{eqnarray} 
In this case, $a_i = u(x_i)$, for every $i=1,...,N+1$. 
We observe that the approximation of the first derivative of $u$ at  $x=x_i$  is:
\begin{eqnarray}
\dfrac{du}{dx}(x_i) \approx \dfrac{du_N}{dx}(x_i) = \dfrac{dP_N}{dx}(x_i) = \sum_{j=1}^{N+1} a_j\> \dfrac{dL_j}{dx}(x_i) = \left(\dfrac{2}{b-a} \right) \sum_{j=1}^{N+1} a_j\> \dfrac{\widehat{L}_j}{dy}(y_i).
\end{eqnarray} 
Thus, the pseudo-spectral derivative, which we will denote as $D_N$, is given by:
 \begin{eqnarray}\label{Matrizprimeraderivada}
(D_N)_{i\>j} = \left(\dfrac{2}{b-a}\right) \dfrac{d\widehat{L}_j}{dy}(y_i), & &1\leq i,j \leq N+1.
\end{eqnarray} Furthermore, the $m$-th pseudo-spectral derivative of $u$, denoted by $D_N^{(m)}$, can be computed as \begin{eqnarray}
D_N^{(m)} = D_N^m = \underbrace{D_N \cdot \cdot \cdot D_N}_{m-times}.
\end{eqnarray} In particular,  \begin{eqnarray}
D_N^{(2)} = D_N^2,  & \hbox{and}  & D_N^{(4)} = D_N^4. 
\end{eqnarray} In the case $a=-1$ and $b=1$, computationally practical methods for deriving the entries of $D_N$ can be found, for instance, in  \cite{GHO} and in \cite{GT}, where explicit formulas are given.\par 
In next section, we will operate on both matrices $D_N^{(2)}$ and $D_N^{(4)}$ in order to generate new matrices in which each type of boundary condition is incorporated into both of them.  

\subsection{Second order differentiation matrices in one dimension}\label{s2do}
In this section, we build second order differentiation matrices enforcing either Dirichlet, Neumann or Robin boundary conditions.  

\subsubsection{Nonhomogeneous Dirichlet boundary conditions} 
Suppose that $u(x)$ satisfies the nonhomogeneous boundary conditions, $u(a)=\alpha$ and $u(b)=\beta$, where $\alpha, \beta \in \mathbb{R}$. Therefore, $u_N(x_{N+1})=u_N(a)=P_N(a)=u(a)=\alpha$ and  $u_N(x_1)=u_N(b)=P_N(b)=u(b)=\beta$. In this case, we approximate the second order derivatives of $u$ at the interior points $x_i$, $i=2,...,N$  as follows 
\begin{eqnarray}
\dfrac{d^2 u}{dx^2}(x_i) \, \approx \, \dfrac{d^2 u_N}{dx^2}(x_i) \, =  \, \beta \> (D_N^{(2)})_{i\>1}  + \sum_{j=2}^{N}(D_N^{(2)})_{i\>j} \> u(x_j) + \alpha \> (D_N^{(2)})_{i\>N+1}.
\end{eqnarray}


First, to describe our further results precisely some notation are required: the matrix ${\widehat{D}_D},$ 
  \begin{eqnarray}
 \left( {\widehat{D}_D}\right)_{i\>j} \, = \,(D_N^{(2)})_{i+1 \> j+1}, & &\quad 1 \leq i,j \leq N-1
  \end{eqnarray}   
   the vector $\overrightarrow{W_D},$  
\begin{eqnarray}
  (\overrightarrow{W_D})_i \,= \,\beta \> (D_N^{(2)})_{i+1\>1}  +  \alpha \> (D_N^{(2)})_{i+1\>N+1}, & & \quad 1 \leq i,j \leq N-1. 
\end{eqnarray}
and the affine transformation $ T_D : \mathbb{R}^{N-1} \rightarrow \mathbb{R}^{N-1}$ 
 \begin{eqnarray}
T_D (\vec{u}_{N-1} ) \, = \, {\widehat{D}_D} \> \vec{u}_{N-1} + \overrightarrow{W_D},
\end{eqnarray}
where $\vec{u}_{N-1}=(u(x_2),...,u(x_N))$, which discretizes the second order derivative on $(a,b)$ subjected to Dirichlet conditions.

\subsubsection{Nonhomogeneous Neumann boundary conditions}
Consecutively, if we  enforce $u'(a) = \alpha$ and $u'(b)=\beta$,  the values of $u_N(a)$ and $u_N(b)$ can be obtained from $\dfrac{du_N}{dx}(x_1)= \beta$, \, $\dfrac{du_N}{dx}(x_{N+1})= \alpha$ as follows: 
 \begin{eqnarray}\label{sistema1}
\left( \begin{array}{cllll}
(D_N)_{N+1 \> 1} & (D_N)_{N+1 \> N+1}\\
(D_N)_{1 \> 1} & (D_N)_{1 \> N+1}
\end{array}\right) \left( \begin{array}{cllll}
u(x_1) \\ u(x_{N+1}) \end{array}\right) & = & \left( \begin{array}{cllll}
\alpha \\ \beta \end{array}\right) - \sum_{j=2}^{N} \left( \begin{array}{cllll}
(D_N)_{N+1 \> j} \\ (D_N)_{1 \> j} \end{array}\right) \> u(x_j).
\end{eqnarray}
Furthermore,  if we introduce the notation \begin{eqnarray}
Q = \left( \begin{array}{cllll}
(D_N)_{N+1 \> 1} & (D_N)_{N+1 \> N+1}\\
(D_N)_{1 \> 1} & (D_N)_{1 \> N+1}
\end{array}\right), & H= \left( \begin{array}{cllll}
\alpha \\ \beta \end{array}\right), & 
\end{eqnarray} 
and
\begin{eqnarray} 
G_j =  \left( \begin{array}{cllll}
(D_N)_{N+1 \> j} \\ (D_N)_{1 \> j} \end{array}\right), &  j=2, \cdots, N, 
\end{eqnarray} 
the formula (\ref{sistema1}) can be rewritten as: \begin{equation}\label{FMsistema1}
Q \>  \left( \begin{array}{cllll}
u(x_1) \\ u(x_{N+1}) \end{array}\right) = H -  \sum_{j=2}^{N} G_j \> u(x_j). 
\end{equation} Therefore,  \begin{equation}\label{Newmnnrecover}
 \left( \begin{array}{cllll}
u(x_1) \\ u(x_{N+1}) \end{array}\right) = Q^{-1}\> H -  \sum_{j=2}^{N} Q^{-1}\>G_j \> u(x_j).
\end{equation} Particularly, we demonstrate the non singularity of the matrix $Q$  in following  proposition. \begin{proposition}
For each integer $N>1$, the matrix $Q$ is nonsingular.
\end{proposition} \begin{proof} 
Note that  $$det(Q) = \dfrac{(2N^2+1)^2-9}{9(b-a)^2}.$$  Hence, $det(Q)\neq 0$ for each integer $N>1$.
\end{proof}

\medskip
\noindent  Consequently, using (\ref{Newmnnrecover}), the pseudo spectral approximation of the second derivative of $u$ at the interior points $x=x_i$ is given by 
 \begin{equation*}
\begin{split}
\dfrac{d^2u}{dx^2}(x_i) & \approx  \dfrac{d^2 u_N}{dx^2}(x_i) \, =     \sum_{j=2}^{N} \left( (D_N^{(2)})_{i\>j} -  (D_N^{(2)})_{i\>1} F_{1}^{Q^{-1}G_j} -  (D_N^{(2)})_{i\>N+1} F_{2}^{Q^{-1}G_j} \right) \> u(x_j) \\ &\quad + (D_N^{(2)})_{i\>1}  F_{1}^{Q^{-1}H} +  (D_N^{(2)})_{i\>N+1}   F_{2}^{Q^{-1}H}, \>\>\>\>\> i=2,...,N.
\end{split}
\end{equation*} Here, $F_i^A$ stands for the $i$-th row of the matrix $A$.
Therefore, we define the matrix $\widehat{D}_{Ne}$ and the vector  $\overrightarrow{W_{Ne}}$ as follows:
\begin{eqnarray}
(\widehat{D}_{Ne})_{i \> j}  =  (D_N^{(2)})_{i+1\>j+1} -  (D_N^{(2)})_{i+1\>1} F_{1}^{Q^{-1}G_{j+1}} -  (D_N^{(2)})_{i+1\>N+1} F_{2}^{Q^{-1}G_{j+1}}, \quad 1\leq i,j \leq N-1
\end{eqnarray}
and 
\begin{eqnarray}
(\overrightarrow{W_{Ne}})_i & = & F_{1}^{Q^{-1}H}  \> (D_N^{(2)})_{i+1 \> 1} +  F_{2}^{Q^{-1}H} \>  (D_N^{(2)})_{i+1 \> N+1}, \quad 1\leq i \leq N-1, 
\end{eqnarray}

\medskip 
\noindent which allows us to define the discretization of the second order derivative on $(a, b)$ subjected to  Neumann conditions through the affine transformation $T_{Ne} : \mathbb{R}^{N-1} \rightarrow \mathbb{R}^{N-1},$ 
\begin{eqnarray}
T_{Ne}(\vec{u}_{N-1}) & = & \widehat{D}_{Ne} \>\vec{u}_{N-1} + \overrightarrow{W_{Ne}},
\end{eqnarray}
being $\vec{u}_{N-1}=(u(x_2),...,u(x_N))$.

\subsubsection{Nonhomogeneous Robin boundary conditions} 
Hereunder, we suppose that \begin{equation}\label{robincond}
\left\{\begin{array}{clllll}
\alpha u(a) - \beta u'(a) = g_1\\
\alpha u(b) + \beta u'(b) = g_2,
\end{array}\right.
\end{equation}
\bigskip
\noindent where \,$\alpha, \beta , g_1, g_2 \in \mathbb{R}$ and $\alpha \beta > 0$. Therefore, the values $u_N(x_{N+1})=u_N(a)$ and $u_N(x_1)=u_N(b)$ satisfy
 \begin{eqnarray}\label{sistema2}
\left( \begin{array}{cllll}
-\beta (D_N)_{N+1 \> 1} & \alpha-\beta (D_N)_{N+1 \> N+1}\\
\alpha + \beta (D_N)_{1 \> 1} & \beta (D_N)_{1 \> N+1}
\end{array}\right) \left( \begin{array}{cllll}
u(x_1) \\ u(x_{N+1}) \end{array}\right) =  \left( \begin{array}{cllll}
g_1 \\ g_2 \end{array}\right) - \sum_{j=2}^{N} \beta \left( \begin{array}{cllll}
-(D_N)_{N+1 \> j} \\ (D_N)_{1 \> j} \end{array}\right) \> u(x_j).
\end{eqnarray} Thereupon, if we set \begin{eqnarray*}
Q_R =\left( \begin{array}{cllll}
-\beta (D_N)_{N+1 \> 1} & \alpha-\beta (D_N)_{N+1 \> N+1}\\
\alpha + \beta (D_N)_{1 \> 1} & \beta (D_N)_{1 \> N+1}
\end{array}\right) , & H_R= \left( \begin{array}{cllll}
g_1 \\ g_2 \end{array}\right), & 
\end{eqnarray*}
and 
\begin{eqnarray*}
G_{R,j} = \beta \left( \begin{array}{cllll}
-(D_N)_{N+1 \> j} \\ (D_N)_{1 \> j} \end{array}\right) \quad j=2,...,N,
\end{eqnarray*}
the formula (\ref{sistema2}) can be rewritten as \begin{equation}\label{FMsistema2}
Q_R \>  \left( \begin{array}{cllll}
u(x_1) \\ u(x_{N+1}) \end{array}\right) = H_R -  \sum_{j=2}^{N} G_{R,j} \> u(x_j).
\end{equation}


\noindent Thus,  \begin{equation}\label{Robinrecover}
 \left( \begin{array}{cllll}
u(x_1) \\ u(x_{N+1}) \end{array}\right) = Q_R^{-1}\> H_R -  \sum_{j=2}^{N} Q_R^{-1}\>G_{R,j} \> u(x_j).
\end{equation}
\bigskip
 The invertibility of the matrix $Q_R$ is guaranteed thanks to the following proposition. 
  \begin{proposition}
For each integer $N>1$, the matrix $Q_R$ is nonsingular.  
\end{proposition} \begin{proof}(Proof by contradiction)\par
\noindent We observe that 
 \begin{eqnarray}\label{detrobin}
det(Q_R) & = & \dfrac{9\beta^2-[3\alpha(b-a)+\beta(2N^2+1)]^2}{9(b-a)^2}.
\end{eqnarray} Therefore, if $det(Q_R)=0$ for some integer $N_0>1$ yields  $$3\alpha (b-a) + \beta (2N_0^2+1)= \pm 3|\beta|.$$ Nevertheless,  the above equality does not hold because $\alpha \beta >0$. Thus $det(Q_R) \neq 0$ for each integer $N>1$.  
\end{proof}

As we have argued in the previous sections and using  (\ref{Robinrecover}), we can approximate the second derivative of $u$ at the interior points as:
 \begin{equation*}
\begin{split}
\dfrac{d^2u}{dx^2}(x_i) &\approx \dfrac{d^2u_N}{dx^2}(x_i) =  \sum_{j=2}^{N} \left( (D_N^{(2)})_{i\>j} -  (D_N^{(2)})_{i\>1} F_{1}^{Q_R^{-1}G_{R,j}} -  (D_N^{(2)})_{i\>N+1} F_{2}^{Q_R^{-1}G_{R,j}} \right) \> u(x_j)\\ &\quad +  (D_N^{(2)})_{i\>1}  F_{1}^{Q_R^{-1}H_R} +  (D_N^{(2)})_{i\>N+1}   F_{2}^{Q_R^{-1}H_R}, \quad \quad 2\leq i \leq N.
\end{split}
\end{equation*}  

\noindent Consequently, we define  the matrix $\widehat{D}_{R}$ and the vector $\overrightarrow{W_R}$  whose entries are: \begin{eqnarray}
(\widehat{D}_{R})_{i \> j}  =  (D_N^{(2)})_{i+1\>j+1} -  (D_N^{(2)})_{i+1\>1} F_{1}^{Q_R^{-1}G_{R,j+1}} -  (D_N^{(2)})_{i+1\>N+1} F_{2}^{Q_R^{-1}G_{R,j+1}}, \quad 1 \leq i,j \leq N-1 
\end{eqnarray} and \begin{eqnarray}
(\overrightarrow{W_R})_i& = & F_{1}^{Q_R^{-1}H_R}  \> (D_N^{(2)})_{i+1 \>  1} +  F_{2}^{Q_R^{-1}H_R} \>  (D_N^{(2)})_{i+1 \> N+1}, \quad 1 \leq i \leq N-1.
\end{eqnarray} 

Lastly, we discretize the second order derivative  on $(a, b)$ subjected to  Robin conditions through  the affine transformation $T_{R} : \mathbb{R}^{N-1} \rightarrow \mathbb{R}^{N-1}$ defined by \begin{equation}
T_{R} (\vec{u}_{N-1}) = \widehat{D}_{R} \> \vec{u}_{N-1} + \overrightarrow{W_R},
\end{equation} being $\vec{u}_{N-1}=(u(x_2),...,u(x_N))$. 


\subsection{Fourth order differentiation matrix in one dimension}\label{s4to}
In this section, the previous ideas are extended in order to discretize the biharmonic problem. In this case, we suppose that 
\begin{eqnarray}\label{fronteracuartoorden}
u(a)=\alpha_1, \; u(b)=\beta_1, \; u'(a)=\alpha_2 \; \hbox{ and } \; u'(b)=\beta_2.
\end{eqnarray}
Therefore, if we  assume that 
\begin{eqnarray}\label{fronteracuartoordenN}
u_N(a)=\alpha_1, \; u_N(b)=\beta_1, \; u'_N(a)=\alpha_2 \; \hbox{ and } \; u'_N(b)=\beta_2,  
\end{eqnarray}
we find that  
\begin{equation}\label{sistema3}
\begin{split}
\left( \begin{array}{cllll}
(D_N)_{N+1 \> 2} & (D_N)_{N+1 \> N}\\
(D_N)_{1 \> 2} & (D_N)_{1 \> N}
\end{array}\right) \left( \begin{array}{cllll}
u(x_2) \\ u(x_{N}) \end{array}\right)&  =  \left( \begin{array}{cllll}
\alpha_2 \\ \beta_2 \end{array}\right) - \beta_1  \left( \begin{array}{cllll}
(D_N)_{N+1 \> 1} \\ (D_N)_{1\>1} \end{array}\right) \\ & \quad  - \alpha_1  \left( \begin{array}{cllll}
(D_N)_{N+1 \>N+1} \\ (D_N)_{1\>N+1} \end{array}\right) 
  - \sum_{j=3}^{N-1} \left( \begin{array}{cllll}
(D_N)_{N+1 \> j} \\ (D_N)_{1 \> j} \end{array}\right) \> u(x_j).
\end{split}
\end{equation} 
Consequently, if we introduce the notation 
 \begin{eqnarray*}
Q_{BH} & = \left( \begin{array}{cllll}
(D_N)_{N+1 \> 2} & (D_N)_{N+1 \> N}\\
(D_N)_{1 \> 2} & (D_N)_{1 \> N}
\end{array}\right), 
\end{eqnarray*}
\begin{eqnarray*}
  G_{BH,j}&  =  \left( \begin{array}{cllll}
(D_N)_{N+1 \> j} \\ (D_N)_{1 \> j} \end{array}\right),  \quad j=3,...,N-1, \\
\end{eqnarray*} and
\begin{eqnarray*}
 H_{BH} & = & \left( \begin{array}{cllll}
\alpha_2 \\ \beta_2 \end{array}\right) - \beta_1  \left( \begin{array}{cllll}
(D_N)_{N+1 \> 1} \\ (D_N)_{1\>1} \end{array}\right)  - \alpha_1  \left( \begin{array}{cllll}
(D_N)_{N+1 \>N+1} \\ (D_N)_{1\>N+1}.\end{array}\right), 
\end{eqnarray*} 
we can rewrite the formula (\ref{sistema3}) as
 \begin{equation}\label{FMsistema3}
Q_{BH} \>  \left( \begin{array}{cllll}
u(x_2) \\ u(x_{N}) \end{array}\right) = H_{BH} -  \sum_{j=3}^{N-1} G_{BH,j} \> u(x_j),  
\end{equation} 
which implies that  \begin{equation}\label{BArecover}
 \left( \begin{array}{cllll}
u(x_2) \\ u(x_{N}) \end{array}\right) = Q_{BH}^{-1}\> H_{BH} -  \sum_{j=3}^{N-1} Q_{BH}^{-1}\>G_{BH,j} \> u(x_j).
\end{equation} 
Therefore, the values of $u_N(x_2)$ and $u_N(x_{N})$  are deduced from \eqref{BArecover}. 
\par
\bigskip 
The following result establishes the invertibility of the matrix $Q_{BH}$. 
\medskip
\begin{proposition}\label{invQBA}
For every integer $N>1$, the matrix $Q_{BH}$ is non singular. Moreover,\begin{eqnarray*} detQ_{BH} = o(N^5) & as & N \uparrow \infty. \end{eqnarray*} 
\end{proposition} 
\begin{proof} The determinant of the matrix $Q_{BH}$ gives  \begin{eqnarray*}
det(Q_{BH}) & = & \dfrac{64cos(\pi/N)}{(b-a)^2sin^4(\pi/N)}. 
\end{eqnarray*} It is clear that $det(Q_{BH})\neq 0$ for every integer $N>1$.
\end{proof}

On the other hand, we can obtain the discretization of the fourth derivative of $u$ at the interior points $x_i$  as   follows: 
 \begin{equation*}
\begin{split}
\dfrac{d^4u}{dx^4}(x_i) & \approx \dfrac{d^4u_N}{dx^4}(x_i)= \beta_1  (D_N^{(4)})_{i\>1} + \alpha_1 (D_N^{(4)})_{i\>N+1} +  (D_N^{(4)})_{i\>2} F_{1}^{Q_{BA}^{-1}H_{BA}} +  (D_N^{(4)})_{i\>N}  F_{2}^{Q_{BA}^{-1}H_{BA}}\\
& \quad + \sum_{j=3}^{N-1} \left( (D_N^{(4)})_{i\>j} -  (D_N^{(4)})_{i\>2} F_{1}^{Q_{BA}^{-1}G_{BA,j}} -  (D_N^{(4)})_{i\>N} F_{2}^{Q_{BA}^{-1}G_{BA,j}} \right) \> u(x_j), \>\>\>\>\> i=3,...,N-1.
\end{split}
\end{equation*} 

\noindent Thereupon, if we introduce the matrix $\widehat{D}_{BA}$ and the vector  $\overrightarrow{W_{BA}}$:
 \begin{equation*} 
\begin{split} 
(\widehat{D}_{BA})_{i \> j}  =  & (D_N^{(4)})_{i+2\>j+2} -  (D_N^{(4)})_{i+2\>2} F_{1}^{Q_{BA}^{-1}G_{BA,j+2}}  \\ & \quad -  (D_N^{(4)})_{i+2\>N} F_{2}^{Q_{BA}^{-1}G_{BA,j+2}}, \quad \quad \quad 1\leq i,\>j \leq N-3,
\end{split}
\end{equation*} 
 \begin{equation*}
\begin{split}
( \overrightarrow{W_{BA}})_{i}  = & \, \beta_1  (D_N^{(4)})_{i+2\>1} + \alpha_1 (D_N^{(4)})_{i+2\>N+1} +  (D_N^{(4)})_{i+2\>2} F_{1}^{Q_{BA}^{-1}H_{BA}}  \\ 
&  \quad  + (D_N^{(4)})_{i+2\>N}  F_{2}^{Q_{BA}^{-1}H_{BA}}, \quad \quad \quad i=1,...,N-3,  
\end{split}
\end{equation*} we can define the affine transformation  $T_{BA} : \mathbb{R}^{N-3} \rightarrow \mathbb{R}^{N-3}$ as follows: 
\begin{equation}
T_{BA} (\vec{u}_{N-3}) =\widehat{D}_{BA} \> \vec{u_{N-3}} + \overrightarrow{W_{BA}},
\end{equation}
being 
\begin{equation}
 \vec{u}_{N-3}=(u(x_3),...,u(x_{N-1})), 
\end{equation}
which discretizes  the fourth order derivative on $(a, b)$ subjected to  the boundary conditions \eqref{fronteracuartoorden}.


\subsection{General discrete formulation of one dimensional problems}\label{seccionpd} 
In this section, using the approach given in Section \ref{s2do}, we  will provide an unified general discretization of problem (\ref{problemaunidimensional}) for $\gamma=2$. Depending on the type of boundary condition, whether Dirichlet, Neumann or Robin, we write the discretization of (\ref{problemaunidimensional}) as
 \begin{eqnarray}\label{P_2o}
\widehat{D}_{\omega} \,\vec{u} + F(x_2,...,x_N,\vec{u}) + \overrightarrow{W_{\omega}} = 0,
\end{eqnarray} where $\vec{u}=(u(x_2),...,u(x_N))$ and the subscript $\omega \in \{D,Ne,R\}$. \par 
Similarly, using the approach given in Section \ref{s4to}, we discretize the problem  (\ref{problemaunidimensional}), for $\gamma=4$, as   
\begin{eqnarray}\label{p4o}
\widehat{D}_{BH} \, \overrightarrow{u_{BH}} + F(x_3,...,x_{N-1}, \overrightarrow{u_{BH}}) + \overrightarrow{W_{BH}} = 0,
\end{eqnarray} being $\overrightarrow{u_{BH}} = (u(x_3),...,u(x_{N-1}))$. \par

We observe that there are $N-1$ unknowns in the problem \eqref{P_2o}, while problem \eqref{p4o} has $N-3$ unknowns. Moreover, in case that the function $F$ in (\ref{problemaunidimensional}) is linear in the variable $u$, both linear systems   \eqref{P_2o} and \eqref{p4o} can be solved isolating the unknowns. Notwithstanding,  if $F$ is a non linear function in the variable $u$, the Newton method has to be used to approximate the value of the unknowns in  (\ref{P_2o}) and  in (\ref{p4o}), respectively.  Finally, the Table \ref{appsol} summarizes how to compute the coefficients $a_i$'s for different types of boundary conditions.

\begin{table} [h] 
\begin{center}
\begin{tabular}{|c|c|c|c|c|c|c|c|c|c|c|c|}
\hline

\multicolumn{3}{|c|}{$u(x) \approx u_N(x)= \sum \limits_ {i=1}^{N+1} a_i\> L_i(x) $} \\
\hline
 Boundary conditions &  Coefficients $a_i$ & Solution's boundary values \\
\hline
Nonhomogeneous Dirichlet     &  $(\beta,u(x_2),...,u(x_N),\alpha)$  & $u(x_1)=\beta$ and  $u(x_{N+1})=\alpha$\\
 $\gamma=2$   & & \\
\hline
Nonhomogeneous Neumann & $(u(x_1),u(x_2),...,u(x_N),u(x_{N+1}))$ & $u(x_1)$ and $u(x_{N+1})$ are \\
 $\gamma=2$& & computed through (\ref{Newmnnrecover}).\\
\hline
Nonhomogeneous Robin & $(u(x_1),u(x_2),...,u(x_N),u(x_{N+1}))$ & $u(x_1)$ and $u(x_{N+1})$ are \\
 $\gamma=2$& & computed through (\ref{Robinrecover}).\\
\hline
Nonhomogeneous Dirichlet & $(\beta_1,u(x_2),u(x_3),...,u(x_{N-1}),u(x_N),\alpha_1)$ & $u(x_1)=\beta_1$, $u(x_{N+1})=\alpha_1$, \\ 
biharmonic
equation & &  $u(x_2)$ and $u(x_{N})$ are\\
 $\gamma=4$ & & computed in (\ref{BArecover}).\\
\hline
\end{tabular}
\caption{ {\small  Here, $\vec{u}$ and $\vec{u}_{BH}$ are calculated solving repectively systems (\ref{P_2o}) and (\ref{p4o}).}}
\label{appsol}
\end{center}
\end{table}


\subsection{Solving nonhomogeneous one dimensional problems}
As an application of discretization, in (\ref{P_2o})  and (\ref{p4o}), four examples are solved: a nonhomogeneous Dirichlet boundary value problem, a non linear Neumann boundary value problem, a Robin boundary value problem and a fourth order boundary value problem.

\begin{example} Let \begin{equation} 
\left\{ \begin{array}{clllll} \label{ejnum1}
\dfrac{d^2u}{dx^2} & =  -\dfrac{12p(2x-1)}{(p+(2x-1)^2)^{5/2}},  & x \in (0,1)\\
u(0) & = 1 &\\
u(1)& = 0. &
\end{array}\right. 
\end{equation}
The exact solution of 
(\ref{ejnum1}) is  \begin{eqnarray}
u(x) = \dfrac{2x-1}{\sqrt{p+(2x-1)^2}} - 
\dfrac{(2\sqrt{p+1}+p+1)(2x-1)}{2(p+1)}.
\end{eqnarray} 
\par 
The left half of Figure \ref{ej1} shows a plot of $u_{500}(x)$ in the case $p=10^{-3}$, and the right half, shows a plot of the corresponding absolute error. The Table \ref{tabla1} collects the $L_2$ and $L_{\infty}$ errors from \cite{LGBM} and the error obtained by using the method proposed in this paper which is substantially smaller than  the corresponding ones for other known methods. 
\begin{figure} [h] 
\begin{center}
\begin{tabular}{lll}
\includegraphics[width=0.5\columnwidth]{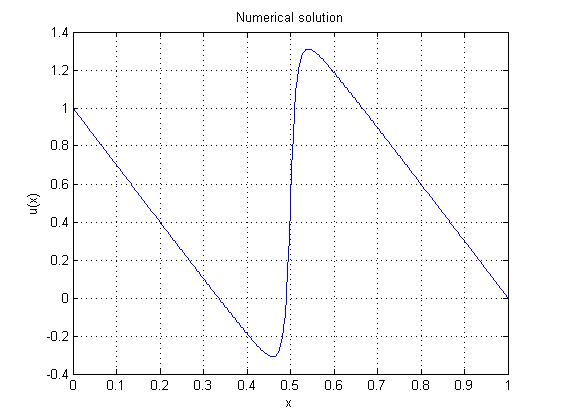} 
& \includegraphics*[width=0.5\columnwidth]{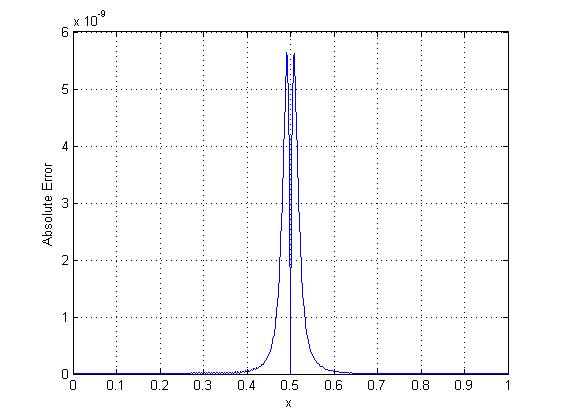}
\end{tabular}
\end{center}
\caption{ { \small (Left) The numerical solution $u_{500}(x)$ of the BVP (\ref{ejnum1}) for $p=10^{-3}$. (Right) The absolute error.}}
\label{ejemplodirichlet}
\label{ej1}
\end{figure}
\begin{table} [h] 
\begin{center}
\begin{tabular}{|c|c|c|c|c|c|c|c|c|c|c|c|}
\hline
 Norm &  Shooting & Finite   &  Finite & Discontinuous & One dimensional  \\
      &  method  & difference & Element & Galerkin & differentiation matrix\\
\hline
 $||\cdot ||_2$ & 1.76e-004 & 9.04e-006 & 1.75e-004 & 1.75e-004 & 1.66e-008 \\
\hline
 $||\cdot||_{\infty}$ & 2.14e-006 & 1.15e-003 & 1.43e-006 & 1.43e-006 & 5.64e-009
\\
\hline
\end{tabular}
\caption{ {\small Errors obtained by taking a grid of 501 collocation points (N=500). The numerical simulations using shooting, finite difference, finite element and discontinuous Galerkin method have been computed in \cite{LGBM}.  } }

\label{tabla1}
\end{center}
\end{table} 
\end{example}

\begin{example}
Let  \begin{equation}\label{n1dexp}
\left\{ \begin{array}{clllll}
\dfrac{d^2u}{dx^2} & = -e^{-2u} & x\in(0,1)\\
u_x(0) & =1 &\\
u_x(1) & =1/2, &
\end{array}\right.
\end{equation} The exact solution of 
(\ref{n1dexp}) is $u(x) = log(1+x)$. If we take $N=20$ and we solve the nonlinear system of equations derived from the discretization of (\ref{n1dexp}) via Newton Method with  a tolerance of $1e-08$, we aill obtain a  maximum error of   $6.9056e-14$. Thereby, the results obtained in \cite{PSP} have  been enhanced. 

\par  The left part of Figure \ref{ejemploneumann} shows the $u_{20}(x)$ and, the right part, shows the plot of the corresponding absolute error.

 \begin{figure}[h]
\begin{center}
\begin{tabular}{lll}
\includegraphics[width=0.5\columnwidth]{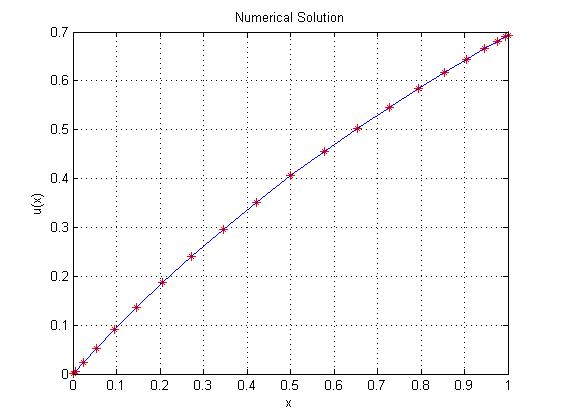} 
& \includegraphics*[width=0.5\columnwidth]{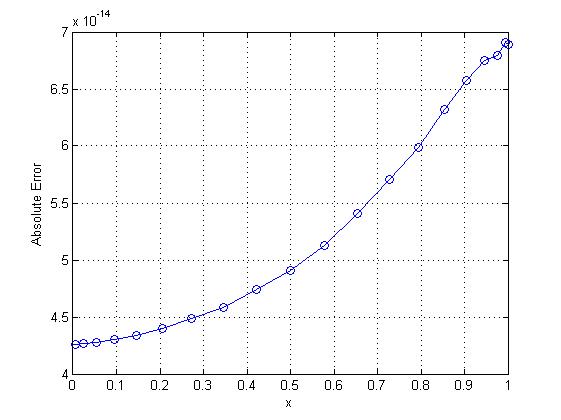}
\end{tabular}
\end{center}
\caption{ { \small (Left) Plot of the approximate $u_{20}(x)$ of the BVP  (\ref{n1dexp}). (Right) Plot of the absolute error.  }}
\label{ejemploneumann}
\end{figure}  
\end{example}

\begin{example} Let 
 \begin{equation}\label{n1dro}
\left\{ \begin{array}{clllll}
-e^x\>\dfrac{d^2u}{dx^2} & = 15 \,cos(4x)-8 \, sin(4x) & x\in(0,2\pi)\\
u(0)-u_x(0) & =2 &\\
u(2\pi)+u_x(2\pi) & =0.  &
\end{array}\right.
\end{equation}  The exact solution of 
(\ref{n1dro}) is $u(x) = cos(4x) \,e^{-x}$.  The left part of Figure \ref{ejemplorobin} shows the plot of  $u_{200}(x)$ while the right part shows the plot of the absolute. The maximum error obtained in this case is $1.6388e-12$.
\begin{figure}[h]
\begin{center}
\begin{tabular}{lll}
\includegraphics[width=0.5\columnwidth]{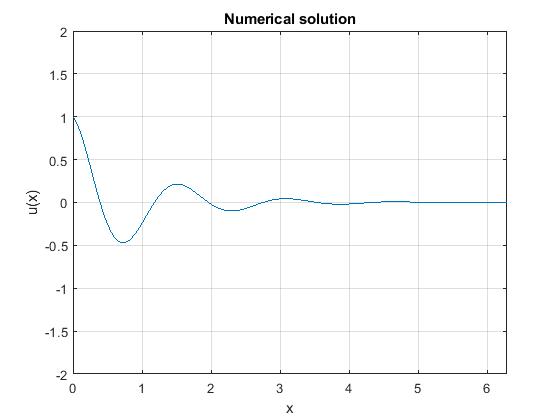} 
& \includegraphics*[width=0.5\columnwidth]{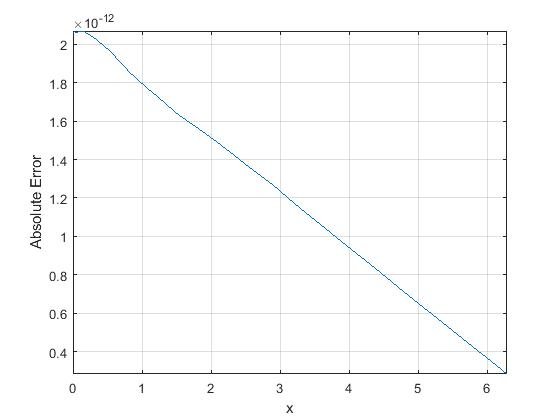}
\end{tabular}
\end{center}
\caption{ { \small (Left) Plot of $u_{200}(x)$ of the BVP (\ref{n1dro}). (Right) Plot of the absolute error. }}
\label{ejemplorobin}
\end{figure}
    
\end{example}
\begin{example}
Let  \begin{equation}\label{n1d4or}
\left\{ \begin{array}{clllll}
0.005 \>\dfrac{d^4u}{dx^4} - u & = 10 & x\in(-1,1)\\
u(-1)=u(1) & =0 &\\
u_x(-1)=u_x(1) & =0. &
\end{array}\right. 
\end{equation} The exact solution of \ref{n1d4or}  is $$u(x) = \dfrac{10 \,sinh(50) \,cos(50x)\,+\,10\, sin(50)\, cosh(50x)}{cosh(50)\, sin(50)+cos(50) \,sinh(50)}-10.$$ Finally, the left part of Figure \ref{ejemploBa} shows a plot of $u_{400}(x)$  for  $N=400$. Concurrently, the right part shows a plot of the absolute error. In this case, the maximum error obtained is $1.7163e-07$.
 
\begin{figure}[h]
\begin{center}
\begin{tabular}{lll}
\includegraphics[width=0.5\columnwidth]{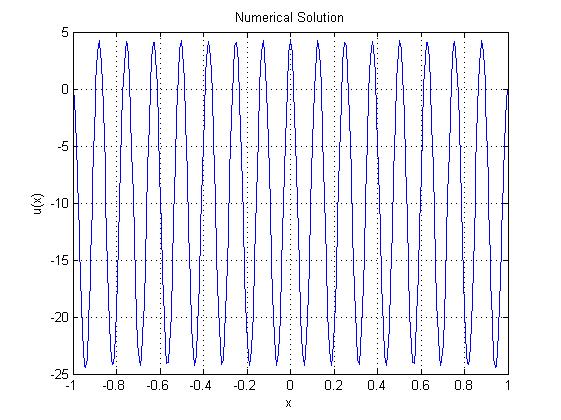} 
& \includegraphics*[width=0.5\columnwidth]{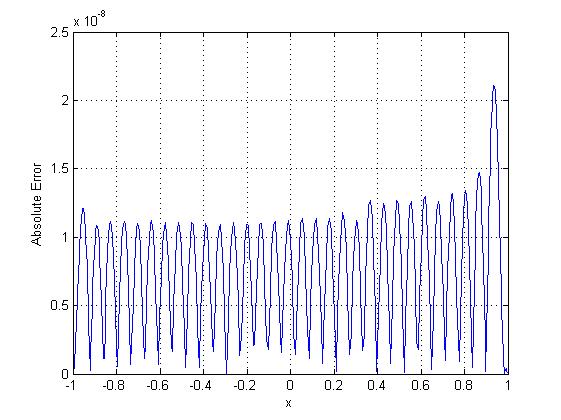}
\end{tabular}
\end{center}
\caption{ { \small (Left) Plot of  $u_{400}(x)$ of the BVP  (\ref{n1d4or}). (Right) Plot of the absolute error.}}
\label{ejemploBa}
\end{figure}  
\end{example}


\section{Polar differentiation matrices}\label{2dimensiones}

In this section, the polar differentiation matrices are defined, for the first time in the literature,  giving a substantial leap with the target of incorporating all type of boundary conditions in the differentiation matrices.  To begin with, we introduce  $u : B_R(0) \rightarrow \mathbb{R}$, which satisfies both \begin{eqnarray}\label{problemabidimensional}
\Delta^{\gamma}u + F(x,y,u) = 0, & & (x,y)\in B_R(0) 
\end{eqnarray}
and certain general boundary conditions where $F$ may be a linear or non linear function. It is worth noting that if $\gamma=1$, we deal  with the Laplacian  \begin{eqnarray}
\Delta(\cdot) = \dfrac{\partial^2}{\partial x^2}(\cdot) + \dfrac{\partial^2}{\partial y^2}(\cdot),
\end{eqnarray} while if $\gamma=2$, we work with  the biharmonic operator  \begin{eqnarray}
\Delta^2(\cdot) = \dfrac{\partial^4}{\partial x^4}(\cdot) + 2 \dfrac{\partial^4}{\partial x^2 \partial y^2}(\cdot) + \dfrac{\partial^4}{\partial y^4}(\cdot).
\end{eqnarray} \par 
Therefore, instead of solving (\ref{problemabidimensional}) in $(x,y)$-space, we consider the following change of variables  \begin{eqnarray}
x=r\>cos(\theta), & & y=r\>sin(\theta).
\end{eqnarray}  Thus, if we see the problem (\ref{problemabidimensional}) in terms of $r$ and $\theta$, we  can rewrite it as: \begin{eqnarray}\label{problemapolar}
\Delta^{\gamma}_{(r,\theta)} u + F(r,\theta,u) = 0 & & (r,\theta) \in [0,R)\times (0,2\pi], 
\end{eqnarray} where the Laplace and the biharmonic operators are respectively: \begin{equation*}
\begin{split}
\Delta_{(r,\theta)} & \coloneqq \dfrac{\partial^2 }{\partial r^2} + \dfrac{1}{r}  \dfrac{\partial }{\partial r} + \dfrac{1}{r^2}  \dfrac{\partial^2 }{\partial \theta^2} \\
\\
\Delta^2_{(r,\theta)} & \coloneqq \dfrac{\partial^4 }{\partial r^4} +  \dfrac{2}{r^2}  \dfrac{\partial^4 }{\partial r^2 \partial \theta^2} +  \dfrac{1}{r^4}  \dfrac{\partial^4 }{\partial \theta^4} +  \dfrac{2}{r}  \dfrac{\partial^3 }{\partial r^3} -  \dfrac{2}{r^3}  \dfrac{\partial^3 }{\partial r \partial \theta^2} -  \dfrac{1}{r^2}  \dfrac{\partial^2 }{\partial r^2} +  \dfrac{4}{r^4}  \dfrac{\partial^2 }{\partial \theta^2} +  \dfrac{1}{r^3}  \dfrac{\partial }{\partial r}.
\end{split}
\end{equation*} 
We note that, to avoid dividing by zero in $\Delta_{(r,\theta)}$ and $\Delta^2_{(r,\theta)}$, we take $N_r+1$ as the number of discretization points  in the $r$-direction, being $N_r$ odd. Moreover, in order to use the symmetry properties in $\theta$, we choose $N_{\theta}$ to be even. \par 
Therefor, we define  \begin{equation}\label{uappr}
u_{\frac{N_r+1}{2},N_{\theta}}(r,\theta) = \sum_{k=1}^{N_r +1} L_k(r) \> P_k(\theta), 
\end{equation}
where $L_k$'s  are the corresponding Lagrange polynomials associated to the nodes \begin{eqnarray}\label{malladoenr}
r_i = R\> y_i=R \,cos \left( \dfrac{(i-1) \> \pi}{N_r}\right), & & 1\leq i \leq N_r+1,
\end{eqnarray} where $\{y_1,...,y_{N_r+1}\}$ are the CGL points, and  \begin{equation}
P_k(\theta)=\sum_{l=1}^{N_{\theta}}a_{k,l} S_{N_{\theta}}(\theta-\theta_l) 
\end{equation} is the  trigonometric interpolants of $u(r_k,\theta)$ at the points $\theta_l$   \begin{eqnarray}\label{malladoentheta}
\theta_l = \dfrac{2\pi \> l }{N_{\theta}}, & & 1\leq l \leq N_{\theta},
\end{eqnarray} 

\noindent and        
 $ S_{N_{\theta}}(\theta)=D_e(N_{\theta}/2, \theta)$ is the   Dirichlet kernel. Thereupon, due to the smoothness of $u(r_k,\theta)$,  Theorem 1.4.2 and Corollary 1.4.7 in \cite{Stenger},   and   
 $h_\theta = \frac {2 \pi}{N_\theta},$ 
 
 \begin{equation}
S_{N_{\theta}}(\theta) = \dfrac{sin\left( \dfrac{N_{\theta} \theta}{2}\right)}{N_{\theta} tan\left( \dfrac{\theta}{2}\right)}.
\end{equation} 
We observe that; from \eqref{uappr} we obtain  $a_{k,l}=u(r_k,\theta_l)$ for all  $1\leq k \leq N_r+1$ and $1 \leq l \leq N_{\theta}$. \par 
On the other hand,  \begin{equation}
u (r_{N_r+2-i},\theta_j)= u(r_i,\theta_{j+\frac{N_\theta}{2}}), \quad 1 \leq j \leq \frac {N_{\theta}}{2}, \quad 1 \leq i \leq \frac{N_r+1}{2}.  
\end{equation} 
Therefore, 
\begin{equation} \label{relacioncoef}
 a_{N_r+2-i,j}= a_{i,j+\frac{N_\theta}{2}}, \quad 1 \leq j \leq \frac {N_{\theta}}{2}, \quad 1 \leq i \leq \frac{N_r+1}{2} .
 \end{equation} 
Hence, we can rewrite  \eqref{uappr} as follows: \begin{equation}
u_{\frac{N_r+1}{2},N_{\theta}}(r,\theta) = \sum_{k=1}^{N_r +1} L_k(r) \> \frac{1}{2\pi} \sum_{\overline{k}=-N_{\theta}/2}^{N_{\theta}/2 \>\>\> \prime}
\hat{a_{k\overline{k}}} \> e^{i\overline{k}\theta} , 
\end{equation} 
where  \begin{equation}
\hat{a_{k\overline{k}}} = \frac{1}{N_{\theta}} \sum_{l=1}^{N_{\theta}} a_{k,l} \> e^{-i\overline{k}\theta_l}. 
\end{equation}
Furthermore,  the prime indicates that the terms $k=\pm N_{\theta}/2$ are multiplied by $1/2$.    
Henceforth, we approximate the solution of (\ref{problemapolar}) by the following sum of finite series \begin{eqnarray}\label{APPSOL2D}
u(r,\theta) \approx u_{\frac{N_r+1}{2},N_{\theta}} (r,\theta) = P_{\frac{N_r+1}{2},N_{\theta}} (r,\theta)\coloneqq  \sum_{k=1}^{N_r+1} \sum_{l=1}^{N_{\theta}} a_{k,l}\, S_{N_{\theta}}(\theta - \theta_l)\, L_{k}(r). 
\end{eqnarray}  
Finally, if we  use  (\ref{relacioncoef}), we can rewrite   $u_{\frac{N_r+1}{2},N_{\theta}}(r,\theta)$  as: \begin{eqnarray} \label{numeincog}
u_{\frac{N_r+1}{2},N_{\theta}}(r,\theta) & = & \sum_{k=1}^\frac{N_r+1}{2} \sum_{l=1}^{N_{\theta}} a_{k,l} \left[ S_{N_{\theta}}(\theta-\theta_l) L_k(r) + S_{N_{\theta}}(\theta-\theta_{l+\frac{N_\theta}{2}}) L_{N_r+2-k}(r) \right].
\end{eqnarray}
 \noindent From the above formula it is deduced that  there are actually $({\frac{N_r+1}{2})N_{\theta}}$ number of unknowns in $u_{\frac{N_r+1}{2},N_{\theta}} (r,\theta)$. 


\subsection{Polar differentiation matrices of the Laplace operator}

 \par In this section, we build the differentiation matrix $D_{\Delta_{(r,\theta)}}$, which discretizes $\Delta_{(r,\theta)}$ in the disk of radius $R$. In order to do this, we  consider the matrix \begin{eqnarray}
(D_{N_r}^{(m)})_{i\>j} = \left(\dfrac{1}{R^m}\right) \dfrac{d^m \widehat{L_j}}{dy^m}(y_i),  & & 1 \leq i,j \leq N_r+1, \>\>\> m=1,2,3,4,... 
\end{eqnarray} and its submatrices defined as  \begin{eqnarray}
(D_1^{(m)})_{i \> j} = (D_{N_r}^{(m)})_{i \> j},  \quad (D_2^{(m)})_{i \> j} = (D_{N_r}^{(m)})_{\frac{N_r+1}{2}+i \>\> N_r+2-j} , & & 1 \leq i,j \leq  \frac{N_r+1}{2}.   
\end{eqnarray} Furthermore, $D_{N_{\theta}}^{(m)}$ denotes the matrix whose corresponding entries are  
 \begin{eqnarray}
(D_{N_{\theta}}^{(m)})_{k\>l} = \dfrac{d^mS_{N_{\theta}}}{d\theta^m}(\theta_k-\theta_l), \quad 1 \leq k,l \leq N_{\theta}. 
\end{eqnarray} 
We observe that $D_{N_{\theta}}^{(m)}$, unlike $D_{N_r}^{(m)}$, cannot be obtained  multiplying $m$-times $D_{N_{\theta}}^{(1)}$.\par 
The next calculations make strong use of (\ref{numeincog}), 
\begin{equation}
\left. \frac{\partial^2 u_{\frac{N_r+1}{2},N_{\theta}}}{\partial \theta^2}\right|_{(r_i,\theta_j)} = \sum_{k=1}^\frac{N_r+1}{2} \sum_{l=1}^{N_{\theta}} a_{k,l} \left( S''_{N_{\theta}}(\theta_j-\theta_l) L_k(r_i) + S''_{N_{\theta}}(\theta_j-\theta_{l+\frac{N_\theta}{2}}) L_{N_r+2-k}(r_i) \right)
=  \sum_{l=1}^{N_{\theta}} a_{i,l}  S''_{N_{\theta}}(\theta_j-\theta_l), 
\end{equation}
\begin{eqnarray*} 
\left. \frac{\partial u_{\frac{N_r+1}{2},N_{\theta}}}{\partial r}\right|_{(r_i,\theta_j)} & = & \left(\dfrac{1}{R}\right) \sum_{k=1}^\frac{N_r+1}{2} \sum_{l=1}^{N_{\theta}} a_{k,l} \left( S_{N_{\theta}}(\theta_j-\theta_l) ) \dfrac{d \widehat{L_k}}{dy}(y_i) + S_{N_{\theta}}(\theta_j-\theta_{l+\frac{N_\theta}{2}})\dfrac{d \widehat{L_{N_r+2-k}}}{dy}(y_i) \right)  \\& &   \\
 & = & \left(\dfrac{1}{R}\right) \sum_{k=1}^\frac{N_r+1}{2} \left( a_{k,j}  \dfrac{d \widehat{L_k}}{dy}(y_i) +  a_{k,j+\frac{N_\theta}{2}} \dfrac{d \widehat{L_{N_r+2-k}}}{dy}(y_i) \right) 
\end{eqnarray*} 
and
 \begin{eqnarray*} 
\left. \frac{\partial^2 u_{\frac{N_r+1}{2},N_{\theta}}}{\partial r^2}\right|_{(r_i,\theta_j)} & = &\left(\dfrac{1}{R^2}\right) \sum_{k=1}^\frac{N_r+1}{2} \sum_{l=1}^{N_{\theta}} a_{k,l} \left[ S_{N_{\theta}}(\theta_j-\theta_l) ) \dfrac{d^2 \widehat{L_k}}{dy^2}(y_i) + S_{N_{\theta}}(\theta_j-\theta_{l+\frac{N_\theta}{2}})\dfrac{d^2 \widehat{L_{N_r+2-k}}}{dy^2}(y_i) \right]. 
\\ & & \\
& = & \left(\dfrac{1}{R^2}\right) \sum_{k=1}^\frac{N_r+1}{2} \left( a_{k,j}  \dfrac{d^2 \widehat{L_k}}{dy^2}(y_i) +  a_{k,j+\frac{N_\theta}{2}} \dfrac{d^2 \widehat{L_{N_r+2-k}}}{dy^2}(y_i) \right).
\end{eqnarray*} 
Finally, taking all the above into account, we define the matrix  $D_{\Delta_{(r,\theta)}}$ as follows:  \begin{equation}\label{laplacedisco}
D_{\Delta_{(r,\theta)}}=(D_{1}^{(2)} + H \> D_{1}^{(1)}) \otimes  \left( \begin{array}{cllll} I & 0 \\ 0 & I \end{array} \right) + (D_{2}^{(2)} + H \> D_{2}^{(1)}) \otimes  \left( \begin{array}{cllll} 0 & I \\ I & 0 \end{array} \right) + H^{2} \otimes D_{N_{\theta}}^{(2)},
\end{equation} 
where I stands for the identity of order $\frac{N_\theta}{2} \times \frac{N_\theta}{2}$ and $H$ is the $ \frac{N_r+1}{2} \times  \frac{N_r+1}{2}$ diagonal matrix $H_{i\>i} = r^{-1}_{i}$ for $i=1,..., \frac{N_r+1}{2}$.\par 

From now on, we use the following notation: \begin{equation} 
\vec{u}^\bold{*}= \sum_{i=1}^\frac{N_r+1}{2} \sum_{j=1}^{N_{\theta}} \, u_{\frac{N_r+1}{2},N_{\theta}} (r_i,\theta_j) \,\, \vec{e}_{N_{\theta}(i-1)+j},  
\end{equation}
where $\vec{e}_1, \vec{e}_2, \cdots,$ 
$\vec{e}_{(\frac{N_r+1}{2})N_{\theta}}$
are the $(\frac{N_r+1}{2})N_{\theta}$ elements of the usual basis of
$\mathbb{R}^{(\frac{N_r+1}{2})N_{\theta}}$.
\par 
Henceforth, if we distinguish the values of $u_{\frac{N_r+1}{2},N_{\theta}}$ in the interior of the disk and we reject  the grid points of the boundary, $  u_{\frac{N_r+1}{2},N_{\theta}}(r_1,\theta_j), \,  1 \leq j \leq  N_{\theta}$, we yield 
  
\begin{equation} 
\vec{u}= \sum_{i=2}^\frac{N_r+1}{2} \sum_{j=1}^{N_{\theta}} \,u_{\frac{N_r+1}{2},N_{\theta}}(r_i,\theta_j)  \,\, \vec{e}_{N_{\theta}(i-2)+j}.
\end{equation}
\par 
We will build in the following subsections the corresponding differentiation matrices of the polar  Laplace operator enforcing, respectively, Dirichlet, Neumann and Robin boundary conditions.  


\subsubsection{Nonhomogeneous Dirichlet boundary conditions} 

Suppose that $u(R,\theta)=f(\theta)$ for $\theta \in (0,2\pi]$, being $f$ a continuous function on $[0,2\pi]$, so as the Dirichlet kernel properties are satisfied. Nonetheless, this condition can be weakened in order to solve the problems arising from applications. Furthermore, the corresponding numerical solution converges. Therefore, if we set \begin{eqnarray}
\vec{f}=  \sum_{j=1}^{N_{\theta}} \, f(\theta_j) \, \vec{e}_{j} & \mbox{and}& \vec{u}^1=  \sum_{j=1}^{N_{\theta}} \, u_{\frac{N_r+1}{2},N_{\theta}}(r_1,\theta_j) \, \, \vec{e}_{j}  
\end{eqnarray}
the boundary condition implies that $\vec{u}^1=\vec{f}$ and  $\Delta_{(r,\theta)}$ evaluated at the interior collocation points $(r_i,\theta_j)$,  for all $i=2,..., \frac{N_r+1}{2}$ and $j=1,...,N_{\theta}$,  can be approximated by  \begin{eqnarray}
\Delta_{(r,\theta)} \; u \approx D_{\Delta_{(r,\theta)}} \vec{u}^* = D^1_{\Delta_{(r,\theta)}} \vec{f}   + D^2_{\Delta_{(r,\theta)}} \vec{u},& & 
\end{eqnarray} where  \begin{eqnarray}
(D^1_{\Delta_{(r,\theta)}})_{i\>j} = (D_{\Delta_{(r,\theta)}})_{N_{\theta}+i \> j}, &  1\leq i \leq  (\frac{N_r-1}{2})N_{\theta}, & 1 \leq j \leq N_{\theta}    
\end{eqnarray} 
and
\begin{eqnarray}
(D^2_{\Delta_{(r,\theta)}})_{i\>j} = (D_{\Delta_{(r,\theta)}})_{N_{\theta}+i \> N_{\theta}+j}, & & 1\leq i,j \leq  (\frac{N_r-1}{2})N_{\theta}.   
\end{eqnarray} 
To finish this section, we define the discretization of $\Delta_{(r,\theta)}$ on $[0,R)\times(0,2\pi]$, subjected to nonhomogeneous Dirichlet boundary conditions through the affine map  \mbox{$T^D_{\Delta_{(r,\theta)}} : \mathbb{R}^{ (\frac{N_r-1}{2})N_{\theta}} \rightarrow \mathbb{R}^{ (\frac{N_r-1}{2})N_{\theta}}$}  which is given by  \begin{eqnarray} \label{laplace1}
T^D_{\Delta_{(r,\theta)}}(\vec{u}) = D^D_{\Delta_{(r,\theta)}} \> \vec{u} + \overrightarrow{W^D_{\Delta_{(r,\theta)}}},
\end{eqnarray} being $ D^D_{\Delta_{(r,\theta)}} =  D^2_{\Delta_{(r,\theta)}}$ and $\overrightarrow{W^D_{\Delta_{(r,\theta)}}} = D^1_{\Delta_{(r,\theta)}} \vec{f}$.


\subsubsection{Nonhomogeneous Neumann boundary conditions}
Now, we suppose that $\frac{\partial u}{\partial r} (R,\theta)= g(\theta)$, for $\theta \in (0,2\pi]$, being $g$ a continuous function on $[0,2\pi]$. Therefore,  
\begin{equation}
\left. \frac{\partial u_{\frac{N_r+1}{2},N_{\theta}}}{\partial r}\right|_{(r_1,\theta_j)}= g(\theta_j), \quad 1 \leq j \leq N_{\theta}. 
\end{equation}
In this case, we must consider the matrix that discretizes $\dfrac{\partial}{\partial r}(\dot)$ on $[0,R]\times (0,2\pi]$:  \begin{equation}\label{matrizp}
P =  D_1^{(1)} \otimes \left( \begin{array}{cllll}
I & 0 \\ 0 & I
\end{array}\right) + D_2^{(1)}  \otimes \left( \begin{array}{cllll}
0 & I \\ I & 0
\end{array}\right),
\end{equation} where  $I$ stands for the $\frac{N_\theta}{2} \times \frac{N_\theta}{2}$ identity matrix. If we highlight  the elements of the matrix $P$ corresponding to $r_1=R$, it  yields to the following matrices:\begin{eqnarray*}
(P_1)_{i\> j}=(P)_{i \> j}, &  &(P_2)_{i \>k}=(P)_{i\> N_{\theta}+k},  
\end{eqnarray*} for   $1 \leq i,j \leq N_{\theta}$ and  $1 \leq k \leq  (\frac{N_r-1}{2})N_{\theta}$. 
Therefore, denoting   
$\vec{g}=  \sum_{j=1}^{N_{\theta}} \, g(\theta_j) \, \vec{e}_{j},$ the nonhomogeneous Neumann boundary conditions implies that
\begin{equation}
\label{CC}
 \vec{g}=  P_1 \vec{u}^1 + P_2 \vec{u}. 
\end{equation}  Finally,  we obtain \begin{equation}\label{CCC}
 \vec{u}^1 = P_1^{-1} \left[\vec{g} - P_2 \>  \vec{u} \right].
\end{equation} The following proposition guarantees the invertibility of the matrix $P_1$.

\begin{proposition}\label{invP1}
For each integer $N_{\theta} \geq 2$ even and each integer $N_r > 1 $, the matrix $P_1$ is nonsingular. 
\end{proposition}\begin{proof}
Note that the matrix $P_1$ has the following form: $$\left(\begin{array}{c|ccc}
(D_{N_r}^{(1)})_{1\>1}\> I & (D_{N_r}^{(1)})_{1\>N_r+1}\> I \\
\hline
 (D_{N_r}^{(1)})_{1\>N_r+1}\> I & (D_{N_r}^{(1)})_{1\>1}\> I
\end{array}\right),$$ where $I$ denotes the identity matrix of order $\frac{N_\theta}{2} \times \frac{N_\theta}{2}$. As the matrix $(D_{N_r}^{(1)})_{1\>1}\> I$ is non singular we obtain  
\begin{eqnarray*}
det(P_1) & = & det((D_{N_r}^{(1)})_{1\>1}\> I)\>det((D_{N_r}^{(1)})_{1\>1}\> I - ((D_{N_r}^{(1)})_{1\>N_r+1}\> I)((D_{N_r}^{(1)})_{1\>1}\> I)^{-1}((D_{N_r}^{(1)})_{1\>N_r+1}\> I)) \\ 
& = & \left[\left((D_{N_r}^{(1)})_{1\>1}\right)^2 - \left((D_{N_r}^{(1)})_{1\>N_r+1}\right)^2\right]^{\frac{N_\theta}{2}} \\
&= & \left( \dfrac{(2N_r^2+1)^2 - 9}{36} \right)^{\frac{N_\theta}{2}}.
\end{eqnarray*} It is clear that $det(P_1)=0$ if and only if $N_r=1$. Thus,  $det(P_1)\neq 0$ for each even integer $N_{\theta} \geq 2$ and each integer $N_r >1$.
\end{proof}

Finally, the approximation of $\Delta_{(r,\theta)}$ at the interior collocation points $(r_i,\theta_j)$  for all $i=2,..., \frac{N_r+1}{2}$ and $j=1,...,N_{\theta}$ is 

\begin{eqnarray*}
\Delta_{(r,\theta)} \; u \approx D_{\Delta_{(r,\theta)}} \vec{u}^* & =  & D^1_{\Delta_{(r,\theta)}} \vec{u}^1   + D^2_{\Delta_{(r,\theta)}} \vec{u} \\  & = & D^1_{\Delta_{(r,\theta)}}\left[  P_1^{-1} \left[\vec{g} - P_2  \vec{u} \right] \right] + D^2_{\Delta_{(r,\theta)}} \vec{u}  \\
 & = & D^1_{\Delta_{(r,\theta)}}  P_1^{-1} \vec{g}  + \left[D^2_{\Delta_{(r,\theta)}} - D^1_{\Delta_{(r,\theta)}} P_1^{-1}  P_2 \right]\vec{u}.
\end{eqnarray*} 

\noindent 
Moreover, we observe that the affine transformation $T^{Ne}_{\Delta_{(r,\theta)}} : \mathbb{R}^{ (\frac{N_r-1}{2})N_{\theta}} \rightarrow \mathbb{R}^{ (\frac{N_r-1}{2})N_{\theta}}$ defined as   \begin{eqnarray} \label{laplace2}
T^{Ne}_{\Delta_{(r,\theta)}}(\vec{u}) = D^{Ne}_{\Delta_{(r,\theta)}} \> \vec{u} + \overrightarrow{W^{Ne}_{\Delta_{(r,\theta)}}},
\end{eqnarray} where $ D^{Ne}_{\Delta_{(r,\theta)}} = D^2_{\Delta_{(r,\theta)}}- D^1_{\Delta_{(r,\theta)}}\>P_1^{-1}\>P_2 $, and $\overrightarrow{W^{Ne}_{\Delta_{(r,\theta)}}} = D^1_{\Delta_{(r,\theta)}} \> P_1^{-1} \vec{g}$, discretizes  $\Delta_{(r,\theta)}$ on $[0,R)\times(0,2\pi]$ subjected to nonhomogeneous Neumann boundary conditions.


\subsubsection{Nonhomogeneous Robin boundary conditions}
In this section, we assume that $a(\theta)\>u(R,\theta)+b(\theta)\>  \frac{\partial u}{\partial r} (R,\theta))=h(\theta)$, where the functions $a,b$ and $h$ are continuous on $[0,2\pi]$ and satisfy $a(\theta)\> b(\theta) > 0$ for all $\theta \in [0,2\pi]$. To describe this boundary conditions, some  notations are required: we denote
$\vec{h}=  \sum_{j=1}^{N_{\theta}} \, h(\theta_j) \, \vec{e}_{j}$ and, given $c\in \{a,b\}$, we denote $M_c$ the diagonal matrix satisfying $(M_c)_{j,j}=c(\theta_j)$ for $j=1,...,N_{\theta}$. Therefore,
  \begin{equation}\label{cccc}
 \vec{h} =  (M_a+M_b\>P_1) \> \vec{u}^1  + M_b\>P_2 \>  \vec{u}
\end{equation} where the matrices $P_1$ and $P_2$ are defined in \eqref{matrizp}. Hence, 
 \begin{eqnarray}\label{Recuprobin2d}
 \vec{u}^1 & = & (M_a + M_b \> P_1)^{-1} \> \vec{h} - (M_a + M_b \> P_1)^{-1} \> M_b \> P_2 \> \vec{u}. 
\end{eqnarray} The invertibility of the matrix $M_a + M_b \> P_1$ is proved in the following proposition: \begin{proposition}\label{invRobin}
The matrix $M_a + M_b \> P_1$ is nonsingular for each integer $N_{\theta} \geq 2$ even and each integer $N_r > 1$.
\end{proposition} \begin{proof}
The proof is based on the following block structure of the matrix $M_a + M_b \> P_1$: 
\medskip
 $$ {\small \left( \begin{array}{ccc|ccc}
a(\theta_1)+b(\theta_1)(D_{N_r}^{(1)})_{1\>1} &        &     & b(\theta_1)(D_{N_r}^{(1)})_{1\>N_r+1} & &  \\
                                & \ddots &     &     & \ddots &               \\
 & &  a(\theta_{\frac{N_\theta}{2}})+b(\theta_{\frac{N_\theta}{2}})(D_{N_r}^{(1)})_{1\>1} & & & b(\theta_{\frac{N_\theta}{2}})(D_{N_r}^{(1)})_{1\>N_r+1} \\
\hline
b(\theta_{\frac{N_\theta}{2}+1})(D_{N_r}^{(1)})_{1\>N_r+1} & & &  a(\theta_{\frac{N_\theta}{2}+1})+b(\theta_{\frac{N_\theta}{2}+1})(D_{N_r}^{(1)})_{1\>1} & & \\
                                & \ddots &     &     & \ddots &               \\
 & & b(\theta_{N_{\theta}})(D_{N_r}^{(1)})_{1\>N_r+1}& & &   a(\theta_{N_{\theta}})+b(\theta_{N_{\theta}})(D_{N_r}^{(1)})_{1\>1} 
\end{array}\right)}.$$
\medskip
where
\medskip
 $$det \left( \begin{array}{cllllll}  
a(\theta_1)+b(\theta_1)(D_{N_r}^{(1)})_{1\>1} &        & \\
                                & \ddots & \\
 & &  a(\theta_{\frac{N_\theta}{2}})+b(\theta_{\frac{N_\theta}{2}})(D_{N_r}^{(1)})_{1\>1} \end{array}\right) = \prod_{j=1}^{\frac{N_\theta}{2}} \left[ a(\theta_j) + b(\theta_j)(D_{N_r}^{(1)})_{1\>1}\right] \neq 0.$$ 
 \medskip
In this direction, the next notation  
 $$M_a + M_b \> P_1 = \left( \begin{array}{c|c}
B & C \\
\hline
D & E 
\end{array}\right)$$
provides us with 
\begin{eqnarray*}
det \,(M_a + M_b \> P_1) & = & det(B) \> det(E-DB^{-1}C) \\
& = & \prod_{j=1}^{\frac{N_\theta}{2}}  \left[ \left(a(\theta_j) + b(\theta_j)(D_{N_r}^{(1)})_{1\>1}\right) \left(a(\theta_{j+\frac{N_\theta}{2}}) + b(\theta_{j+\frac{N_\theta}{2}})(D_{N_r}^{(1)})_{1\>1}\right) \right.\\ 
& & \>\>\>\>\>\>\>\>  \left. -b(\theta_j)b(\theta_{j+\frac{N_\theta}{2}})\left((D_{N_r}^{(1)})_{1\>N_r+1}\right)^2\right].
\end{eqnarray*} Arguing by contradiction and supposing that $det(M_a + M_b \> P_1) = 0$, we yield that there exists $j_0 \in \{1,...,\frac{N_\theta}{2}\}$ such that 
\begin{eqnarray*} 
a(\theta_{j_0})a(\theta_{j_0+\frac{N_\theta}{2}})+a(\theta_{j_0})b(\theta_{j_0+\frac{N_\theta}{2}})(D_{N_r}^{(1)})_{1\>1} + b(\theta_{j_0})a(\theta_{j_0+\frac{N_\theta}{2}})(D_{N_r}^{(1)})_{1\>1} = \\
 \quad  \quad \quad \quad \quad
b(\theta_{j_0})b(\theta_{j_0+\frac{N_\theta}{2}})\left[\left((D_{N_r}^{(1)})_{1\>N_r+1}\right)^2  -\left((D_{N_r}^{(1)})_{1\>1}\right)^2\right].
 \end{eqnarray*}
Notwithstanding, the above equality  cannot be right due to $a(\theta)b(\theta)>0$ for all $\theta \in [0,2\pi]$. Thus $det(M_a + M_b \> P_1)\neq 0$  for each integer $N_{\theta} \geq 2$ even and each integer $N_r > 1$.
\end{proof} 
 Finally, we approximate for all $i=2,..., \frac{N_r+1}{2}$ and $j=1,...,N_{\theta}$,  $\Delta_{(r,\theta)}\>u$  at the collocation points $(r_i,\theta_j)$ as follows
 \medskip
   \begin{eqnarray*}
\Delta_{(r,\theta)} \; u \approx D_{\Delta_{(r,\theta)}} \vec{u}^* &= &  D^1_{\Delta_{(r,\theta)}} \vec{u}^1 + D^2_{\Delta_{(r,\theta)}} \vec{u} \\
& =& D^1_{\Delta_{(r,\theta)}} \left((M_a + M_b \> P_1)^{-1} \> \vec{h} - (M_a + M_b \> P_1)^{-1} \> M_b \> P_2 \> \vec{u} \right) + D^2_{\Delta_{(r,\theta)}} \vec{u} \\
&= & D^1_{\Delta_{(r,\theta)}}\>(M_a+M_b P_1)^{-1}\vec{h}  + \left( D^2_{\Delta_{(r,\theta)}} - D^1_{\Delta_{(r,\theta)}}\>(M_a+M_b P_1)^{-1}\>M_b\> P_2 \right)\> \vec{u}. 
\end{eqnarray*} 
\medskip
Hence, the affine transformation $T^{R}_{\Delta_{(r,\theta)}} : \mathbb{R}^{ (\frac{N_r-1}{2})N_{\theta}} \rightarrow \mathbb{R}^{ (\frac{N_r-1}{2})N_{\theta}}$ defined as \begin{eqnarray} \label{laplace3}
T^{R}_{\Delta_{(r,\theta)}}(\vec{u}) = D^{R}_{\Delta_{(r,\theta)}} \> \vec{u} + \overrightarrow{W^{R}_{\Delta_{(r,\theta)}}},
\end{eqnarray} where $ D^{R}_{\Delta_{(r,\theta)}} =  D^2_{\Delta_{(r,\theta)}}- D^1_{\Delta_{(r,\theta)}}\>(M_a+M_b P_1)^{-1}\>M_b\> P_2$ and, $\overrightarrow{W^{R}_{\Delta_{(r,\theta)}}} =  D^1_{\Delta_{(r,\theta)}}\>(M_a+M_b P_1)^{-1}\> \vec{h}$, discretizes  $\Delta_{(r,\theta)}$ on $[0,R)\times(0,2\pi]$ subjected to nonhomogeneous Robin boundary conditions. 


\subsection{Polar differentiation matrix of the biharmonic operator}

This section addresses a discretization of the biharmonic operator $\Delta^2_{(r,\theta)}$ in the disk of radius $R$. It follows from (\ref{numeincog}), that
\begin{eqnarray*} 
\left. \frac{\partial^4 u_{\frac{N_r+1}{2},N_{\theta}}}{\partial r^4}\right|_{(r_i,\theta_j)} & = &\left(\dfrac{1}{R^4}\right) \sum_{k=1}^\frac{N_r+1}{2} \sum_{l=1}^{N_{\theta}} a_{k,l} \left[ S_{N_{\theta}}(\theta_j-\theta_l) ) \dfrac{d^4 \widehat{L_k}}{dy^4}(y_i) + S_{N_{\theta}}(\theta_j-\theta_{l+\frac{N_\theta}{2}})\dfrac{d^4 \widehat{L_{N_r+2-k}}}{dy^4}(y_i) \right] 
\\ & & \\
& = & \left(\dfrac{1}{R^4}\right) \sum_{k=1}^\frac{N_r+1}{2} \left( a_{k,j}  \dfrac{d^4\widehat{L_k}}{dy^4}(y_i) +  a_{k,j+\frac{N_\theta}{2}} \dfrac{d^4 \widehat{L_{N_r+2-k}}}{dy^4}(y_i) \right),
\end{eqnarray*}

\begin{eqnarray*} 
\left. \frac{\partial^3 u_{\frac{N_r+1}{2},N_{\theta}}}{\partial r^3}\right|_{(r_i,\theta_j)} & = &\left(\dfrac{1}{R^3}\right) \sum_{k=1}^\frac{N_r+1}{2} \sum_{l=1}^{N_{\theta}} a_{k,l} \left[ S_{N_{\theta}}(\theta_j-\theta_l) ) \dfrac{d^3 \widehat{L_k}}{dy^3}(y_i) + S_{N_{\theta}}(\theta_j-\theta_{l+\frac{N_\theta}{2}})\dfrac{d^3 \widehat{L_{N_r+2-k}}}{dy^3}(y_i) \right] 
\\ & & \\
& = & \left(\dfrac{1}{R^3}\right) \sum_{k=1}^\frac{N_r+1}{2} \left( a_{k,j}  \dfrac{d^3\widehat{L_k}}{dy^3}(y_i) +  a_{k,j+\frac{N_\theta}{2}} \dfrac{d^3 \widehat{L_{N_r+2-k}}}{dy^3}(y_i) \right),
\end{eqnarray*} 
\medskip
  \begin{equation}
\left. \frac{\partial^4 u_{\frac{N_r+1}{2},N_{\theta}}}{\partial \theta^4}\right|_{(r_i,\theta_j)} = \sum_{k=1}^\frac{N_r+1}{2} \sum_{l=1}^{N_{\theta}} a_{k,l} \left( S''''_{N_{\theta}}(\theta_j-\theta_l) L_k(r_i) + S''''_{N_{\theta}}(\theta_j-\theta_{l+\frac{N_\theta}{2}}) L_{N_r+2-k}(r_i) \right)
=  \sum_{l=1}^{N_{\theta}} a_{i,l}  S''''_{N_{\theta}}(\theta_j-\theta_l), 
\end{equation}
\medskip
\begin{equation}
\left. \frac{\partial^3 u_{\frac{N_r+1}{2},N_{\theta}}}{\partial r  \, \partial \theta^2}\right|_{(r_i,\theta_j)} = \left(\dfrac{1}{R}\right)\sum_{k=1}^\frac{N_r+1}{2} \sum_{l=1}^{N_{\theta}} a_{k,l} \left( S''_{N_{\theta}}(\theta_j-\theta_l) \dfrac{d\widehat{L_k}}{dy}(y_i) + S''_{N_{\theta}}(\theta_j-\theta_{l+\frac{N_\theta}{2}}) 
\dfrac{d\widehat{L_{N_r+2-k}}}{dy}(y_i)
\right)
\end{equation}
and
\medskip
\begin{equation}
\left. \frac{\partial^4 u_{\frac{N_r+1}{2},N_{\theta}}}{\partial \theta^4}\right|_{(r_i,\theta_j)} = \sum_{k=1}^\frac{N_r+1}{2} \sum_{l=1}^{N_{\theta}} a_{k,l} \left( S''''_{N_{\theta}}(\theta_j-\theta_l) L_k(r_i) + S''''_{N_{\theta}}(\theta_j-\theta_{l+\frac{N_\theta}{2}}) L_{N_r+2-k}(r_i) \right)
=  \sum_{l=1}^{N_{\theta}} a_{i,l}  S''''_{N_{\theta}}(\theta_j-\theta_l). 
\end{equation}
Therefore, rewriting $\Delta^2_{(r,\theta)}$ as
\begin{equation}
\Delta^2_{(r,\theta)}  \coloneqq \dfrac{\partial^4 }{\partial r^4}  
+  \dfrac{2}{r}  \dfrac{\partial^3 }{\partial r^3}
-  \dfrac{1}{r^2}  \dfrac{\partial^2 }{\partial r^2}
+ \dfrac{1}{r^3}  \dfrac{\partial }{\partial r} 
+ \dfrac{2}{r^2}  \dfrac{\partial^4 }{\partial r^2 \partial \theta^2} 
-  \dfrac{2}{r^3}  \dfrac{\partial^3 }{\partial r \partial \theta^2} 
  +  \dfrac{1}{r^4}  \dfrac{\partial^4 }{\partial \theta^4}
   +  \dfrac{4}{r^4}  \dfrac{\partial^2 }{\partial \theta^2}, 
\end{equation}
and concatenating all the above derivates of $u_{\frac{N_r+1}{2},N_{\theta}}$  at the collocation points, we obtain the following expression for the differentiation matrix $D_{\Delta^2_{(r,\theta)}}$ \begin{eqnarray*}\label{biarmonicodisco}
D_{\Delta^2_{(r,\theta)}} & = & \left(D_{1}^{(4)} + 2HD_{1}^{(3)}-H^2D_{1}^{(2)}+H^3D_{1}^{(1)}\right) \otimes  \left( \begin{array}{cllll} I & 0 \\ 0 & I \end{array} \right)\\
& &  +\left(D_{2}^{(4)} + 2HD_{2}^{(3)}-H^2D_{2}^{(2)}+H^3D_{2}^{(1)}\right) \otimes  \left( \begin{array}{cllll} 0 & I \\ I & 0 \end{array} \right) \\ 
& & + \left[ \left(2H^2D_{1}^{(2)}-2H^3D_{1}^{(1)}\right) \otimes  \left( \begin{array}{cllll} I & 0 \\ 0 & I \end{array} \right) +\left(2H^2D_{2}^{(2)}-2H^3D_{2}^{(1)}\right) \otimes  \left( \begin{array}{cllll} 0 & I \\ I & 0 \end{array} \right)\right] \left( I_{ \frac{N_r+1}{2}} \otimes D_{N_{\theta}}^{(2)} \right) \\
& & + H^4 \otimes \left[ D_{N_{\theta}}^{(4)} + 4D_{N_{\theta}}^{(2)} \right] ,
\end{eqnarray*}  where the matrices $H$ and $I$ are defined as in previous sections and in particular, $I_{ \frac{N_r+1}{2}}$ is the $( \frac{N_r+1}{2})\times( \frac{N_r+1}{2})$ identity matrix.

\subsubsection{Nonhomogeneous Dirichlet boundary conditions}
In this case, we assume the following boundary conditions \begin{eqnarray}\label{CC4to2D} 
\mbox{ $u(R,\theta)=f(\theta)$ and $\frac{\partial u}{\partial r} (R,\theta)= g(\theta)$}
\end{eqnarray}
 being both $f$ and $g$ continuous functions on $[0,2\pi]$. Therefore, \begin{eqnarray*}
 \vec{u}^1= \vec{f} &\hbox{   and  }   \left. \frac{\partial u_{\frac{N_r+1}{2},N_{\theta}}}{\partial r}\right|_{(r_1,\theta_j)} =g(\theta_j),  &1\leq j \leq N_{\theta}. 
 \end{eqnarray*}
Hereinafter in this paper, 
\begin{equation} 
\vec{u}^{**}= \sum_{i=3}^\frac{N_r+1}{2} \sum_{j=1}^{N_{\theta}} \, u_{\frac{N_r+1}{2},N_{\theta}}(r_i,\theta_j) \,\, \vec{e}_{N_{\theta}(i-3)+j}.
\end{equation}
and 
\begin{equation} 
\vec{u}^2= \sum_{j=1}^{N_{\theta}} \, u_{\frac{N_r+1}{2},N_{\theta}}(r_2,\theta_j) \,\, \vec{e}_{j}.
\end{equation}

\noindent Likewise, as we are dealing with the biharmonic equation and enforcing two boundary conditions in \eqref{CC4to2D}, we need to define now three submatrices of the matrix P given in \eqref{matrizp}:  \begin{eqnarray*}
(P_1)_{i\>j} = (P)_{i\>j}, & (P_2)_{i,\>j}=(P)_{i\>j+N_{\theta}}, & \mbox{for $1\leq i,j \leq N_{\theta},$}\\
(P_3)_{i\>j} = (P)_{i\>j+2N_{\theta}}, & & \mbox{for $1\leq i \leq N_{\theta}$ and $1 \leq j \leq  (\frac{N_r-3}{2})N_{\theta}$.}
\end{eqnarray*}
\noindent  
Therefore, we obtain from \ref{CC4to2D} that  
 \begin{eqnarray}\label{despu2}
\vec{g} =  P_1 \vec{f} + P_2 \vec{u}^2  + P_3 \vec{u}^{**}.
\end{eqnarray} The following proposition shows the invertivility of the matrix $P_2$. 

\begin{proposition}\label{pppp}
For each integer $N_{\theta} \geq 2$ even and each integer $N_r > 1$, the matrix $P_2$ is nonsingular.
\end{proposition} \begin{proof}
The matrix $P_2$ has the form: $$\left(\begin{array}{c|cc}
\left(D_{N_r}^{(1)}\right)_{1\>2}I & \left(D_{N_r}^{(1)}\right)_{1\>N_r}I\\
\hline
\left(D_{N_r}^{(1)}\right)_{1\>N_r}I & \left(D_{N_r}^{(1)}\right)_{1\>2}I
\end{array}\right),$$ where $I$ is the $\frac{N_\theta}{2} \times \frac{N_\theta}{2}$ identity matrix. Now, since $ \left(D_{N_r}^{(1)}\right)_{1\>N_r}I$ is nonsingular the determinant of $P_2$ gives:  \begin{eqnarray*}
det(P_2) & = & det((D_{N_r}^{(1)})_{1\>2}\> I)\>det((D_{N_r}^{(1)})_{1\>2}\> I - ((D_{N_r}^{(1)})_{1\>N_r}\> I)((D_{N_r}^{(1)})_{1\>2}\> I)^{-1}((D_{N_r}^{(1)})_{1\>N_r}\> I)) \\ 
& = & \left[\left((D_{N_r}^{(1)})_{1\>2}\right)^2 - \left((D_{N_r}^{(1)})_{1\>N_r}\right)^2\right]^{\frac{N_\theta}{2}}.
\end{eqnarray*} Therefore, $det(P_2) \neq 0$ due to $y_2 \neq y_{N_r}$. Thus, $P_2$ is nonsingular, for each integer $N_{\theta} \geq 2$ even and each integer $N_r > 1$.
\end{proof} 

Accordingly to  the Proposition (\ref{pppp}), we can isolate  $\vec{u}^2$ from (\ref{despu2}) as follows \begin{eqnarray}\label{recupBA2d}
 \vec{u}^2 = P_2^{-1} \vec{g} - P_2^{-1} P_1 \vec{f} - P_2^{-1}P_3 \vec{u}^{**}.
\end{eqnarray} 

\noindent Thus, the approximation of $\Delta^2_{(r,\theta)} u$ on $[0,R)\times(0,2\pi]$ at the interior collocation points $(r_i,\theta_j)$, for $i=3,...,N_r+1$ and $j=1,...,N_{\theta}$, remains as 
\begin{eqnarray*}
\Delta^2_{(r,\theta)} u \approx D_{\Delta^2_{(r,\theta)}} \vec{u}^*&=& 
D^1_{\Delta^2_{(r,\theta)}} \vec{u}^1 + D^2_{\Delta^2_{(r,\theta)}} \vec{u}^2
+ D^3_{\Delta^2_{(r,\theta)}} \vec{u}^{**}, \\
&=& 
D^1_{\Delta^2_{(r,\theta)}} \vec{f} + D^2_{\Delta^2_{(r,\theta)}} \left( P_2^{-1} \vec{g} - P_2^{-1} P_1 \vec{f} - P_2^{-1}P_3 \vec{u}^{**} \right)
+ D^3_{\Delta^2_{(r,\theta)}} \vec{u}^{**}, 
\\  & = & \left(D^1_{\Delta^2_{(r,\theta)}}-D^2_{\Delta^2_{(r,\theta)}}P_2^{-1}P_1\right)\vec{f} + D^2_{\Delta^2_{(r,\theta)}} P_2^{-1}\vec{g} + \left(D^3_{\Delta^2_{(r,\theta)}}-D^2_{\Delta^2_{(r,\theta)}} P_2^{-1}P_3\right)\vec{u}^{**}, 
\end{eqnarray*}

\noindent where $D^i_{\Delta^2_{(r,\theta)}}$'s are the submatrices of $D_{\Delta^2_{(r,\theta)}}$ whose entries are respectively \begin{eqnarray*}
\left(D^1_{\Delta^2_{(r,\theta)}}\right)_{i\>j}=\left(D_{\Delta^2_{(r,\theta)}}\right)_{i+2N_{\theta}\>j} ,& \left(D^2_{\Delta^2_{(r,\theta)}}\right)_{i\>j}=\left(D_{\Delta^2_{(r,\theta)}}\right)_{i+2N_{\theta}\>j+N_{\theta}},&
\left(D^3_{\Delta^2_{(r,\theta)}}\right)_{i\>j}=\left(D_{\Delta^2_{(r,\theta)}}\right)_{i+2N_{\theta}\>k+2N_{\theta}},  
\end{eqnarray*} for $1 \leq i \leq (\frac{N_r-1}{2})N_{\theta}$, \; $1 \leq j \leq N_{\theta}$ and $1 \leq k \leq (\frac{N_r-3}{2})N_{\theta}$.

In closing, we define the affine map $T^{BH}_{\Delta^2_{(r,\theta)}} : \mathbb{R}^{ (\frac{N_r-3}{2})N_{\theta}} \rightarrow \mathbb{R}^{ (\frac{N_r-3}{2})N_{\theta}}$  as 
\begin{eqnarray} \label{laplace4}
T^{BH}_{\Delta^2_{(r,\theta)}}(\vec{u}^{**}) = D^{BH}_{\Delta^2_{(r,\theta)}} \> \vec{u}^{**} + \overrightarrow{W^{BH}_{\Delta^2_{(r,\theta)}}},
\end{eqnarray} where  $ D^{BH}_{\Delta^2_{(r,\theta)}} = D^3_{\Delta^2_{(r,\theta)}}-D^2_{\Delta^2_{(r,\theta)}}P_2^{-1}P_3$ and $\overrightarrow{W^{BH}_{\Delta^2_{(r,\theta)}}} = \left(D^1_{\Delta^2_{(r,\theta)}}-D^2_{\Delta^2_{(r,\theta)}}P_2^{-1}P_1\right)\vec{f} + D^2_{\Delta^2_{(r,\theta)}}P_2^{-1}\vec{g}$. This affine map $T^{BH}_{\Delta^2_{(r,\theta)}}$ discretizes  $\Delta^2_{(r,\theta)}$ on $[0,R)\times(0,2\pi]$ subjected to nonhomogeneous Dirichlet boundary conditions. 


\subsection{General discrete formulation for the Laplace equation and the biharmonic equations in a disk} We describe two general abstract formulations of the problem (\ref{problemabidimensional}). In particular, in the case of the Laplace operator, from (\ref{laplace1}), (\ref{laplace2}) and (\ref{laplace3}) it follows that  
\begin{eqnarray}\label{sistema12d}
N^{\omega}_{\Delta_{(r,\theta)}} \coloneqq D^{\omega}_{\Delta_{(r,\theta)}}\vec{u} + F(r_2,...,r_{ \frac{N_r+1}{2}},\theta_1,...,\theta_{N_{\theta}},\vec{u}) + \overrightarrow{W^{\omega}_{\Delta_{(r,\theta)}} } =0.
\end{eqnarray}  The superscript $\omega \in \{D, Ne, R\}$ refers to the type of boundary conditions, i.e. Dirichet, Neumann or Robin, respectively.

Likewise, in the case of the biharmoic operator,  (\ref{laplace4}) yields \begin{eqnarray}\label{sistema22d}
N^{BH}_{\Delta^2_{(r,\theta)}} \coloneqq D^{BH}_{\Delta^2_{(r,\theta)}}\vec{u}^{**}+ F(r_3,...,r_{ \frac{N_r+1}{2}},\theta_1,...,\theta_{N_{\theta}},\vec{u}^{**}) + \overrightarrow{W^{BH}_{\Delta^2_{(r,\theta)}} } =0. 
\end{eqnarray} Moreover, we note that the system (\ref{sistema22d}) has $( \frac{N_r-3}{2}) N_{\theta}$ unknowns while the system (\ref{sistema12d}) has $ (\frac{N_r-1}{2})N_{\theta}$ unknowns.  As discussed above in Section \ref{seccionpd}, depending on the linearity or nonlinearity of the function $F$, different standard methods can be used to solve either  (\ref{sistema12d}) or  (\ref{sistema22d}) systems. Futher on, in  Table \ref{tabla2} we summarized the values $a_{k,l}=u(r_k,\theta_{l})$  of the approximate $u_{\frac{N_r+1}{2},N_{\theta}}(r,\theta)$ at the collocation points depending on each type of boundary condition. 

\begin{table}[ht]
\begin{center}
\begin{tabular}{|c|c|c|c|c|c|c|c|c|c|c|c|}
\hline
\multicolumn{3}{|c|}{$u_{\frac{N_r+1}{2},N_{\theta}}(r,\theta)  =  \sum \limits_{k=1}^\frac{N_r+1}{2} \sum \limits_{l=1}^{N_{\theta}} a_{k,l} \left[ S_{N_{\theta}}(\theta-\theta_l) L_k(r) + S_{N_{\theta}}(\theta-\theta_{l+\frac{N_\theta}{2}}) L_{N_r+2-k}(r) \right]$} \\
\hline
 Boundary conditions &  Coefficients $a_{k,l}$ & Solution's boundary values \\
\hline
 Nonhomogeneous Dirichlet     &  $a_{k,l}=u(r_k,\theta_l)$ and $a_{1,l}=f(\theta_l),$  & $u(r_1,\theta) = f(\theta)$\\
               &  for $2\leq k\leq  \frac{N_r+1}{2}$ and $1\leq l \leq N_{\theta}.$                           & \\ 
\hline
Nonhomogeneous Neumann &  $a_{k,l}=u(r_k,\theta_l),$  & $u(r_1,\theta_l)$ for $1 \leq l \leq N_{\theta}$ \\
& for $1\leq k\leq  \frac{N_r+1}{2}$ and $1\leq l \leq N_{\theta}.$ & are computed through (\ref{CCC}).\\
\hline
 Nonhomogeneous Robin &  $a_{k,l}=u(r_k,\theta_l),$  & $u(r_1,\theta_l)$ for $1 \leq l \leq N_{\theta}$ \\
& for $1\leq k\leq  \frac{N_r+1}{2}$ and $1\leq l \leq N_{\theta}$ & are computed through (\ref{Recuprobin2d}).\\
\hline
Nonhomogeneous Dirichlet &  $a_{k,l}=u(r_k,\theta_l)$ and $a_{1,l}=f(\theta_l)$  & $u(r_1,\theta) = f(\theta),$ and \\
biharmonic equation &  for $2\leq k\leq \frac{N_r+1}{2}$ and $1\leq l \leq N_{\theta}.$                           &  $u(r_2,\theta_l)$ for $1 \leq l \leq N_{\theta}$\\
& & are computed through (\ref{recupBA2d}).\\
\hline
\end{tabular}
\caption{ {\small  As seen above, $\vec{u}$ and $\vec{u}^{**}$ are calculated solving systems (\ref{sistema12d}) and (\ref{sistema22d}), respectively.}}
 \label{tabla2}
\end{center}
\end{table}


\subsection{Solving numerical examples of the Laplace and the biharmonic nonhomogeneous equations}
In this section, six numerical examples are developed, three correspond to the Laplace operator and three to the biharmonic operator. The developed simulations are computed using the differentiation matrices calculated in the previous subsections, either for $\Delta_{(r,\theta)}$ or $\Delta^2_{(r,\theta)}$.  

\begin{example}
The actual solution of the Laplace equation subjected to nonhomogeneous Dirichlet 
 boundary conditions \begin{eqnarray} \label{ej2d1} 
\left\{ \begin{array}{clllll} 
\Delta_{(r,\theta)}u & = 0 & (r,\theta) \in [0,1)\times[0,2\pi)\\
u(1,\theta) & = sin^3\theta & \theta \in [0,2\pi),
\end{array}\right.
\end{eqnarray}
 is given by $u(r,\theta)=0.75\, - 0.25\, r^3\, sin(3\,\theta).$ The approximate solution, in the case $\frac{N_r+1}{2}=28$ 
  and $N_{\theta}=60$, takes the form 
\begin{equation} u_{28,60}(r,\theta)  =  \sum_{k=1}^{28} \sum_{l=1}^{60} a_{k,l} \left[ S_{N_{\theta}}(\theta-\theta_l) L_k(r) + S_{N_{\theta}}(\theta-\theta_{l+\frac{N_\theta}{2}}) L_{N_r+2-k}(r) \right]. 
\end{equation}
The Figure \ref{dirichlet2dplots} shows a graph of the computed solution $u_{28,60}(r,\theta)$  and its absolute error. Nonetheless, in the Table \ref{tabladirichlet}, we list the maximum errors for different values of $\frac{N_r+1}{2}$ and $N_{\theta}$. 

\begin{figure}[h]
\begin{center}
\begin{tabular}{lll}
\includegraphics[width=0.5\columnwidth]{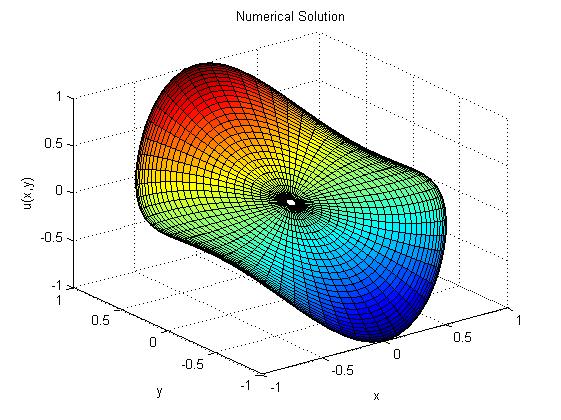} 
& \includegraphics*[width=0.5\columnwidth]{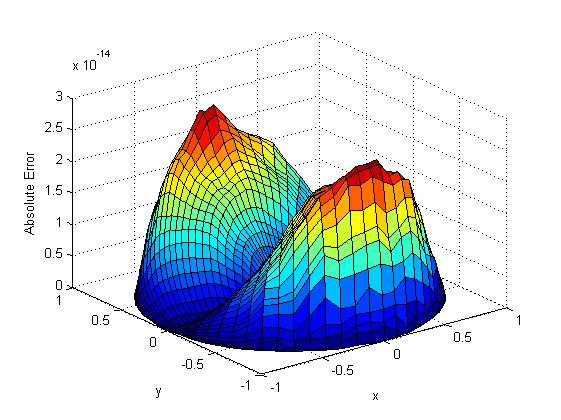}
\end{tabular}
\end{center}
\caption{ { \small (Left)  Computed solution $u_{28,60}(r,\theta)$ of (\ref{ej2d1}). (Right) The absolute error.}}
\label{dirichlet2dplots}
\end{figure}

\begin{table}[ht]
\begin{center}
\begin{tabular}{|c|c|c|c|c|c|c|c|c|c|c|c|}
\hline
 & Simulation 1 & Simulation 2 & Simulation 3 & Simulation 4 & Simulation 5\\
\hline
$\frac{N_r+1}{2}$  & 11 & 28 & 51 & 51 & 101\\
\hline
$N_{\theta}$ & 30 & 60 & 40 & 60 & 100\\
\hline
Maximum Error & 4.5242e-15 & 2.6887e-14 & 1.7447e-13 & 5.9730e-14 & 6.6391e-14\\
\hline
\end{tabular}
\caption{ {\small Maximum errors in the Dirichlet problem for different choices of $\frac{N_r+1}{2}$ and $N_{\theta}$.}}
\label{tabladirichlet}
\end{center}
\end{table}

\end{example}

\begin{example}
Consider the exact solution of  \begin{eqnarray} \label{ej2d2}
\left\{ \begin{array}{clllll}
\Delta_{(r,\theta)}u - u & = r\,(2+5\,sin^2\theta)-r^3sin^2\theta & (r,\theta) \in [0,1)\times[0,2\pi)\\
u_r(1,\theta) & = 3\,sin^2\theta & \theta \in [0,2\pi),
\end{array}\right. 
\end{eqnarray}
which is given by  $u(r,\theta) = r^3\,sin^2\theta$. Once again, we have computed the maximum errors for different values of $\frac{N_r+1}{2}$ and $N_{\theta}$, which are collected in the Table \ref{tablaneumann}. The Figure \ref{neumann2dplots} shows the plots of the numerical solution and the absolute error for $\frac{N_r+1}{2}=31$ and $N_{\theta}=50$. The maximum error obtained with this choice can be found in the Table \ref{tablaneumann}. 
 
\end{example}

\begin{example} The nonlinear Fisher equation \begin{eqnarray} \label{ej2d3}
\left\{ \begin{array}{clllll}
-\Delta_{(r,\theta)}u & = 3 u-u^2 & (r,\theta) \in [0,1)\times[0,2\pi)\\
u(1,\theta)+u_r(1,\theta) & = 3 & \theta \in [0,2\pi),
\end{array}\right. 
\end{eqnarray}
This equation has its unique positive solution given by $u \equiv 3$. In Table \ref{tablarobin} we have  collected some maximum errors, computed for different values of $\frac{N_r+1}{2}$ and $N_{\theta}$. A plot of the numerical solution and the absolute error can be found in the Figure \ref{robin2dplots} for $\frac{N_r+1}{2}=31$ and $N_{\theta}=50$. The maximum error obtained with this choice can be found in the Table \ref{tablarobin}. 
\end{example}
\begin{figure}[h]
\begin{center}
\begin{tabular}{lll}
\includegraphics[width=0.5\columnwidth]{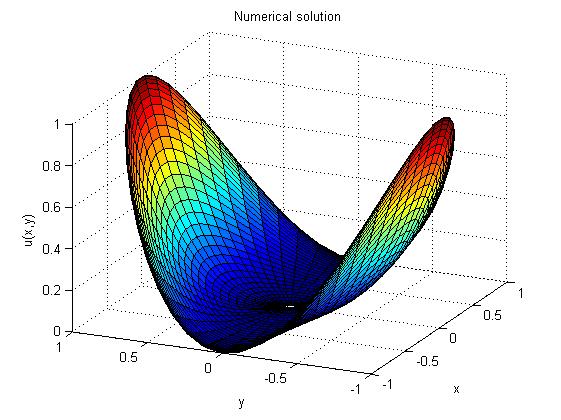} 
& \includegraphics*[width=0.5\columnwidth]
{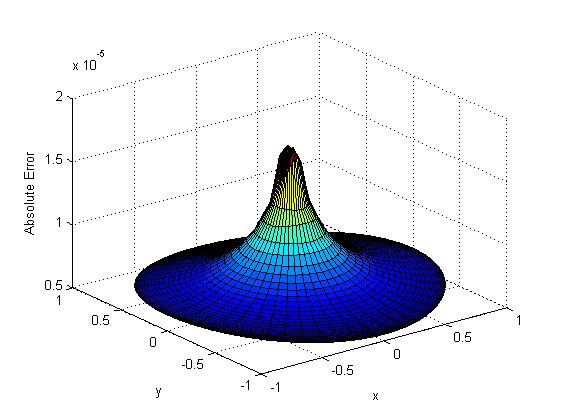}
\end{tabular}
\end{center}
\caption{ { \small (Left)  Computed solution $u_{31,50}(r,\theta)$ of (\ref{ej2d2}). (Right) The absolute error.}}
\label{neumann2dplots}
\end{figure}

\begin{table}[h]
\begin{center}
\begin{tabular}{|c|c|c|c|c|c|c|c|c|c|c|c|}
\hline
 & Simulation 1 & Simulation 2 & Simulation 3 & Simulation 4 \\
\hline
$\frac{N_r+1}{2}$ & 31 &  51 & 101 & 151\\
\hline
$N_{\theta}$ & 50 & 40 & 40 & 40\\
\hline
Maximum Error & 2.4389e-04 & 9.5423e-05 & 2.5333e-05 & 1.1491e-05 \\
\hline
\end{tabular}
\caption{ {\small Maximum errors in the Neumann problem (\ref{ej2d2}) for different choices  of $\frac{N_r+1}{2}$ and $N_{\theta}$.}}
\label{tablaneumann}
\end{center}
\end{table}

\begin{figure}[h!]
\begin{center}
\begin{tabular}{lll}
\includegraphics[width=0.5\columnwidth]{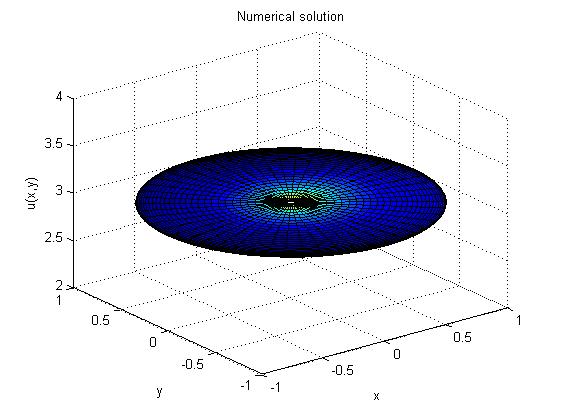} 
& \includegraphics*[width=0.5\columnwidth]{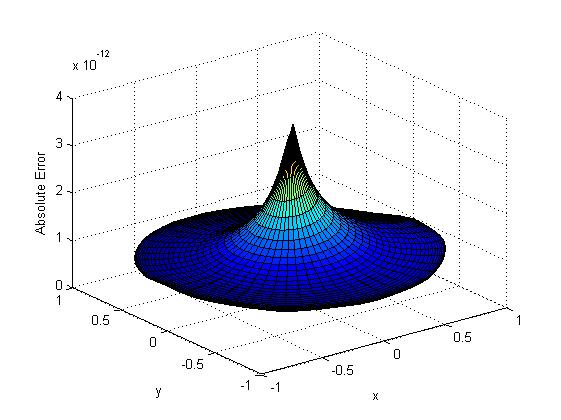}
\end{tabular}
\end{center}
\caption{ { \small (Left)  Computed solution $u_{28,60}(r,\theta)$ of (\ref{ej2d3}). (Right) The absolute error.}}
\label{robin2dplots}
\end{figure}

\begin{table}[ht]
\begin{center}
\begin{tabular}{|c|c|c|c|c|c|c|c|c|c|c|c|}
\hline
 & Simulation 1 & Simulation 2 & Simulation 3 & Simulation 4 & Simulation 5\\
\hline
$\frac{N_r+1}{2}$ & 11 & 31 & 31 & 101 & 101\\
\hline
$N_{\theta}$ & 40 & 50 & 100 & 30 & 50\\
\hline
Maximum Error & 2.9168e-12 & 4.2902e-11 & 1.1023e-10 & 1.1723e-09 & 1.7640e-09\\
\hline
\end{tabular}
\caption{ {\small Robin problem (\ref{ej2d3}): Maximum error for different values of $\frac{N_r+1}{2}$ and $N_{\theta}$.}}
\label{tablarobin}
\end{center}
\end{table}


\begin{example}
Consider the biharmonic equation   \begin{eqnarray} \label{BH2d}
\left\{ \begin{array}{clllll}
\Delta^2_{(r,\theta)}u  & = 0 & (r,\theta) \in [0,1)\times[0,2\pi)\\
u(1,\theta) & = -0.25 & \theta \in [0,2\pi),\\
u_r(1,\theta) & = -0.5(1+cos(\theta)) & \theta \in [0,2\pi),
\end{array}\right. 
\end{eqnarray}
whose exact solution is $u(r,\theta)=0.25(1-r^2)(1+r\>cos(\theta))-0.25$. We compute the approximate  solution $u_{\frac{N_r+1}{2}, N_{\theta}}(r,\theta)$, in the case $\frac{N_r+1}{2}=62$ and $N_{\theta}=40$. Here, the maximum error between $u_{62,40}(r,\theta)$ and the exact solution $u(r,\theta)$ is $8.1766e-04$. The Figure \ref{BH} shows a plot of the computed solution $u_{62,40}(r,\theta)$. 
\begin{figure}[h]
\begin{center}
\includegraphics[width=0.6\columnwidth]{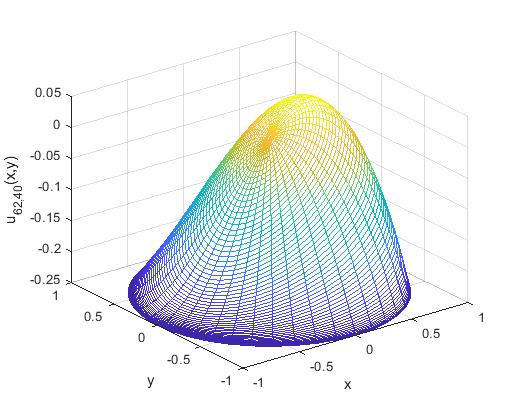} 
\end{center}
\caption{ { \small   Computed solution $u_{62,40}(x,y)$ of (\ref{BH2d}) in cartesian coordinates. }}
\label{BH}
\end{figure}

\end{example}


\begin{example}
The exact solution of the biharmonic equation   \begin{eqnarray*} \label{BH2d1}
\left\{ \begin{array}{clllll}
\Delta^2_{(r,\theta)}u  & = 0 & (r,\theta) \in [0,1)\times[0,2\pi)\\
u(1,\theta) & = cos(2\theta) & \theta \in [0,2\pi),\\
u_r(1,\theta) & = 0 & \theta \in [0,2\pi),
\end{array}\right. 
\end{eqnarray*}
is  $u(r,\theta)=(2r^2-r^4)cos(2\theta))$. We compute the approximate  solution $u_{\frac{N_r+1}{2}, N_{\theta}}(r,\theta)$, in the case $\frac{N_r+1}{2}=48$ and $N_{\theta}=40$. Here, the maximum error between $u_{48,40}(r,\theta)$ and the exact solution $u(r,\theta)$ is $1.9727e-04$ improving the result in \cite{YUref}. The Figure \ref{BH1} shows a plot of the computed solution $u_{48,40}(x,y)$. 
\begin{figure}[h]
\begin{center}
\includegraphics[width=0.6\columnwidth]{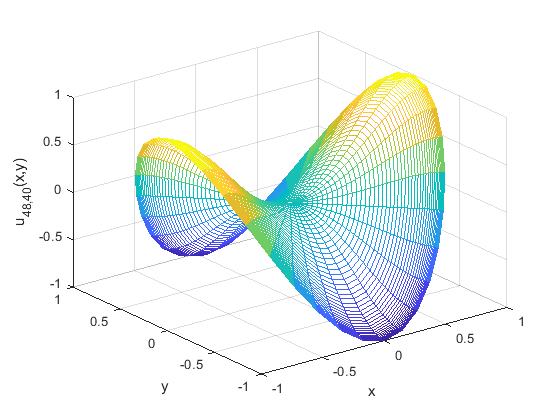} 
\end{center}
\caption{ { \small   Computed solution $u_{48,40}(x,y)$ of (\ref{BH2d1}) in cartesian coordinates. }}
\label{BH1}
\end{figure}

\end{example}

\begin{example}
In closing, we consider the biharmonic equation   \begin{eqnarray*} \label{BH2d2}
\left\{ \begin{array}{clllll}
\Delta^2_{(r,\theta)}u  & = 45cos(\theta) & (r,\theta) \in [0,1)\times[0,2\pi)\\
u(1,\theta) & = cos(2\theta) & \theta \in [0,2\pi),\\
u_r(1,\theta) & = 4cos(2\theta) & \theta \in [0,2\pi),
\end{array}\right. 
\end{eqnarray*}
whose exact solution is $u(r,\theta)=r^4\>cos(2\theta)$. We calculate numerically the approximate  solution $u_{\frac{N_r+1}{2}, N_{\theta}}(r,\theta)$, in the case $\frac{N_r+1}{2}=33$ and $N_{\theta}=60$. Here, the maximum error between $u_{33,60}(r,\theta)$ and the exact solution $u(r,\theta)$ is $4.9969e-05$. The Figure \ref{BH2} shows a plot of $u_{33,60}(x,y)$. 

\end{example}
\bigskip

\begin{figure}[h!]
\begin{center}
\includegraphics[width=0.6\columnwidth]{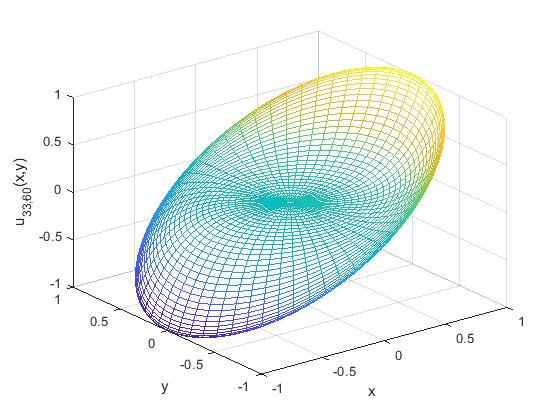} 
\end{center}
\caption{ { \small   Computed solution $u_{33,60}(x,y)$ of (\ref{BH2d1}) in cartesian coordinates. }}
\label{BH2}
\end{figure}

\section{Conclusions}
The differentiation matrices deduced in this paper will be of the utmost importance, since a large number of equations, which model a broad range of applications such as Navier-Stokes equations, are now subjected to nonhomogeneous boundary conditions and could hereafter be  solved with an efficient, simple and direct method. The construction and calculation of each differentiation matrix has been a cumbersome procedure. Nonetheless, we have  provided a clear development for each equation and for its respective Dirichlet, Neumann and Robin nonhomogeneous boundary conditions. 
\par
This paper is completed with a collection of linear and nonlinear numerical examples,  whose solutions exhibit a  spectral accuracy, underling, once again, the advantages of using collocation methods. Now more than ever, no lifting is needed.  

\section{Acknowledgments}
 This work has been partially supported by the Ministry of Economy and Competitiveness of Spain under Research Grant MAT2015-65899-P.


\end{document}